\documentclass[12pt,a4paper,reqno]{amsart}

\usepackage{graphicx}
\usepackage{overpic}
\usepackage{color}
\usepackage[T1]{fontenc}

\usepackage[margin=2.9cm]{geometry}

\usepackage[utf8]{inputenc}
\usepackage{amsmath}
\usepackage{amsfonts}
\usepackage{enumerate}
\usepackage{array}

\newcommand{\R}{\ensuremath{\mathbb{R}}}

\newcommand{\CF}{\ensuremath{\mathcal{F}}}
\newcommand{\CC}{\ensuremath{\mathcal{C}}}
\newcommand{\CO}{\ensuremath{\mathcal{O}}}
\newcommand{\A}{\ensuremath{\mathcal{A}}}

\newcommand{\al}{\ensuremath{\mathcal{\alpha}}}

\newcommand{\T}{\theta}
\newcommand{\e}{\varepsilon}
\newcommand{\p}{\partial}
\newcommand{\ov}{\overline}

\DeclareMathOperator{\Span}{Span}

\newtheorem{thm}{Theorem}[section]

\newtheorem{cor}[thm]{Corollary}
\newtheorem{lem}[thm]{Lemma}
\newtheorem{prop}[thm]{Proposition}

\newtheorem{rem}[thm]{Remark}

\newtheorem{example}[thm]{Example}

\newcolumntype{L}{>{\displaystyle}l}
\newcolumntype{C}{>{\displaystyle}c}
\newcolumntype{R}{>{\displaystyle}r}

\allowdisplaybreaks

\title[Hilbert number in piecewise quadratic systems]{New lower bound for the Hilbert number in piecewise quadratic differential systems}

\author[L. P. C. da Cruz] {Leonardo P. C. da Cruz}
\address{Departament de Matem\`{a}tiques, Universitat Aut\`{o}noma de Barcelona, 08193 Bellaterra, Barcelona,
	Catalonia, Spain}
\email{leonardo@mat.uab.cat}

\author[D.~Novaes]{Douglas D. Novaes}
\address{Departamento de Matem\'{a}tica, Universidade Estadual de Campinas,
 Rua S\'{e}rgio Buarque de Holanda, 651, Cidade Universit\'{a}ria Zeferino Vaz,
 13083--859, Campinas, SP, Brazil} \email{ddnovaes@ime.unicamp.br}

\author[J.~Torregrosa]{Joan Torregrosa}
\address{Departament de Matem\`{a}tiques,
Universitat Aut\`{o}noma de Barcelona, 08193 Bellaterra, Barcelona,
Spain} \email{torre@mat.uab.cat}

\subjclass[2010]{Primary: 37G15, 37C27. Secundary: 37G10, 34C07}

\keywords{Non-smooth differential system, limit cycles in piecewise quadratic differential systems, first and second order perturbations of isochronous quadratic systems, Hilbert number for piecewise quadratic differential systems}

\begin{document}

\begin{abstract} We study the number of limit cycles bifurcating from a piecewise quadratic system. All the differential systems considered are piecewise in two zones separated by a straight line. We prove the existence of 16 crossing limit cycles in this class of systems. If we denote by $H_p(n)$ the extension of the Hilbert number to degree $n$ piecewise polynomial differential systems, then $H_p(2)\geq 16.$ As fas as we are concerned, this is the best lower bound for the quadratic class. Moreover, all the limit cycles appear in one nest bifurcating from the period annulus of some isochronous quadratic centers.
\end{abstract}

\maketitle

\section{Introduction} \label{se:1}
Consider the class of polynomial differential systems of degree $n.$ The maximum number of isolated periodic orbits, the so-called \emph{limit cycles}, that a polynomial differential system of degree $n$ can have is called \emph{Hilbert number}, $H(n).$ It is well known that linear systems have no limit cycles, then $H(1)=0.$ For $n=2,$ the problem of estimating $H(2)$ has been studied intensively during the last century. Lower bounds for $H(2)$ can be given by providing concrete examples of polynomial differential systems of degree $2.$ Up to now, the best result was given by Shi in \cite{Songling1980}, where he proved the existence of a quadratic system with 4 limit cycles in configuration $(3,1),$ that is $H(2)\geq4.$ We call by $M(n)$ the maximum number of limit cycles bifurcating from a singular point as a degenerate Hopf bifurcation. Clearly, $M(n)$ is a lower bound for $H(n).$ Bautin showed in \cite{Bautin1954} that $M(2)=3;$ in \cite{Zol1995,Zol2016}, \.Zo\l{}\k{a}dek proved that $M(3)\geq11;$ a simpler proof was provided by Christopher in \cite{Chr2006}. For $n=3,$ Li, Liu, and Yang proved in \cite{LiLiuYang2009} that $H(3)\geq13.$ 

In the last few years there has been an increasing interest in piecewise smooth systems. This interest has been mainly motivated by their wider range of application in modeling real phenomena (see, for instance, \cite{AcaBonBro2011,BerBudChaKow2008}). In this paper we shall deal with the following class of piecewise vector fields
\begin{equation}\label{eq:0}
Z(x,y)= \left\{
\begin{array}{ccc}
Z^+(x,y),\quad h(x,y)>0,\\
Z^-(x,y),\quad h(x,y)<0,\\
\end{array}\right.
\end{equation}
where $Z^{\pm}=(X^{\pm},Y^{\pm})$ are smooth vector fields and $h:\R^2 \rightarrow \R$ is a $\mathcal C^1$ function for which $0$ is a regular value. In the above vector field, the discontinuity curve and the regions where $Z^\pm$ are denoted by $\Sigma=h^{-1}(0)$ and $\Sigma^\pm=\{(x,y)\in \R^2:\pm h(x,y)>0\},$ respectively. The local trajectories of $Z$ on $\Sigma$ was stated by Filippov in \cite{Fil1988} (see Figure~\ref{fi:filipov}). The points on $\Sigma$ where both vectors fields simultaneously point outward or inward from $\Sigma$ define the \emph{escaping} ($\Sigma^e$) and \emph{sliding region} ($\Sigma^s$), respectively. The interior of its complement on $\Sigma$ defines the \emph{crossing region} ($\Sigma^c$), and the boundary of these regions is constituted by tangential points of $Z^{\pm}$ with $\Sigma.$
\begin{figure}[h]
\begin{center}
\begin{overpic}{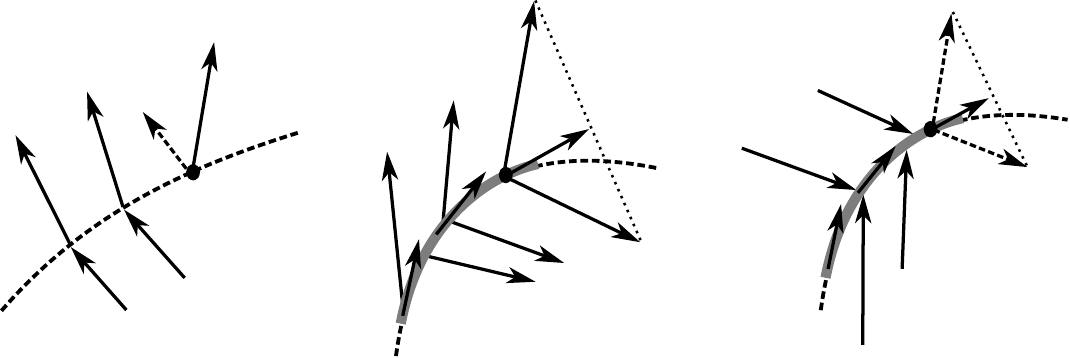}
\put(29,20){$\Sigma$}
\put(62.5,16){$\Sigma$}
\put(101,21){$\Sigma$}
\put(0,10){$\Sigma^+$}
\put(32,10){$\Sigma^+$}
\put(72,10){$\Sigma^+$}
\put(2,2){$\Sigma^-$}
\put(40,0){$\Sigma^-$}
\put(83,2){$\Sigma^-$}
\put(16,14){$p$}
\put(46,14){$p$}
\put(86,18){$p$}
\put(9,30){$Z^+(p)=Z^{\pm}(p)$}
\put(38,30){$Z^+(p)$}
\put(94,14){$Z^+(p)$}
\put(9,24){$Z^-(p)$}
\put(57,7){$Z^-(p)$}
\put(77,30){$Z^-(p)$}
\put(56,21){$Z^{\pm}(p)$}
\put(93,24){$Z^{\pm}(p)$}
\end{overpic}
\end{center}
\caption{Definition of the vector field on $\Sigma$ following Filippov's convention in the sewing, escaping, and sliding regions.}\label{fi:filipov}
\end{figure}
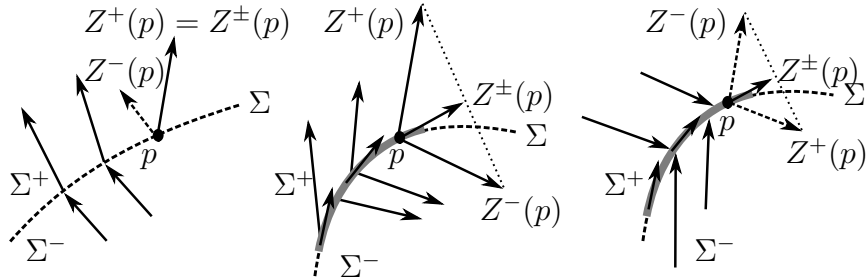
Let $Z^{\pm}h$ denote the derivative of the function $h$ in the direction of the vector $Z^{\pm}$ that is, $Z^{\pm}h(p)=\langle \nabla h(p), Z^{\pm}(p)\rangle.$ Notice that $p\in\Sigma^c$ provided that $Z^+h(p)\cdot Z^-h(p) > 0,$ $p\in\Sigma^e\cup\Sigma^s$ provided that $Z^+h(p)\cdot Z^-h(p) < 0,$ and $p$ in $\Sigma$ is a tangential point of $Z^{\pm}$ provided that $Z^+h(p)Z^-h(p)=0.$ We say that $p \in \Sigma$ is a \emph{singularity} of $Z,$ if $p$ is either a tangential point or a singularity of $Z^+$ or $Z^-.$ We call $p\in \Sigma$ an \emph{invisible fold} of 
$Z^+$ (resp. $Z^-$) if $p$ is a tangential point of $Z^+$ (resp. $Z^-$) and $(Z^+)^2h(p)<0$ (resp. $(Z^-)^2h(p)>0$). 

Analogously to the smooth case, we denote by $H^c_p(n)$ the maximum number of crossing limit cycles that piecewise polynomial differential systems of degree $n$ admit when the curve of discontinuity is a straight line. We also denote by $M^c_p(n)$ the maximum number of crossing limit cycles bifurcating from a singular point or sliding set. Up to now, for piecewise linear systems in two zones separated by a straight line, there are no examples with more than 3 limit cycles. An example with 3 limit cycles was firstly detected numerically in \cite{HuanYang2010} by Huan and Yang. Later, it was analytically proved by Llibre and Ponce in \cite{LliPon2012}. The existence of 3 limit cycles was also obtained from perturbations of a center. For instance, Buzzi et al. in \cite{BuzPesTor2013} obtained 3 limit cycles after a seventh order piecewise linear perturbation of a linear center, and Llibre et al. in \cite{LliNovTei2015a} obtained the same result through a first order perturbation of a piecewise linear center. We may also quote Freire et al. \cite{TorrePonceFre2014}. Consequently, for piecewise linear systems in two zones separated by a straight line we have $H_p^c(1)\geq3.$ 

The averaging theory of order five for studying piecewise perturbations of the linear center was used by Llibre and Tang in \cite{LliTan2016} who provided that $H_p^c(2)\geq8$ and $H^c_p(3)\geq M_p^c(3)\geq 13.$  Recently in \cite{GuoYuChe2018} this number has been improved for piecewise cubic systems providing  $H^c_p(3)\geq 18.$ These are the best results so far for piecewise quadratic and cubic systems in two zones separated by a straight line. Previously, using the averaging theory of first order for studying piecewise perturbations of some quadratic isochronous systems, Llibre and Mereu in \cite{MereuLlibre2014} obtained only 5 limit cycles. Recently, in \cite{CenLiuYanZha2018} the authors study this perturbation problem, only up to first order but for degree $n$. It is worthwhile to say that for quadratic polynomial systems Chicone and Jacobs in \cite{ChiJac1991} proved that at most 2 limit cycles can bifurcate from any period annulus.

In this paper we shall use the averaging theory of first and second order to provide better lower bounds for the maximum number of limit cycles that piecewise quadratic systems can have. More specifically, we shall give examples satisfying $M_p^c(2)\geq 16.$ Consequently, $H^c_p(2)\geq M^c_p(2)\geq16.$ Table~\ref{ta:Hilbert} summarizes the results about the Hilbert numbers for lower degree vector fields. 

\begin{thm}\label{thm:Main} 
There exists a piecewise planar quadratic differential system in two zones separated by a straight line with $16$ crossing limit cycles. 
\end{thm}

\begin{table}[h]
\begin{center}
\begin{tabular}{|c|c|c|}
	\hline
 $\deg$ & PVF & PPVF\\
 	\hline
$n=1$ &	$H(1)=0$ & $H^c_p(1)\geq 3$\\
	\hline
$n=2$ & $H(2) \geq {4}$ & $H^c_p(2)\geq16$ \\
	\hline
$n=3$ & $H(3)\geq 13$ & $H^c_p(3)\geq 18$\\ 
	\hline
\end{tabular}\vspace{5mm}
\end{center}
\caption{Summary of Hilbert numbers for polynomial and piecewise polynomial systems of degree $n.$}\label{ta:Hilbert}
\end{table}

In order to prove our main result we shall proceed with a first and second order perturbation analysis of quadratic isochronous centers. 
In \cite{Lou64} the quadratic isochronous centers are classified in four families, namely $S_1,$ $S_2,$ $S_3,$ and $S_4.$ In \cite{MarRouTon1995} their isochronicity properties were proved as well as their linearizations. In this paper we consider the first three classes of centers, which are birational equivalent to the linear one. This property does not hold for $S_4.$ It is proved in \cite{ChaSab1999} that any center of families $S_1,$ $S_2,$ and $S_3$ can be transformed, after a birational change of variables, in one of the following centers:
\begin{equation}\label{eq:10}
S_1:\begin{array}{l}
\left\{\begin{array}{l}
\dot x= -y+x^2-y^2,\\[6pt]
\dot y=x+2xy.
\end{array}
\right.
\,\,
S_2:\left\{\begin{array}{l}
\dot x=-y+x^2,\\[6pt]
\dot y=x+xy.
\end{array}
\right.
\end{array}
\,\,
S_3:\left\{
\begin{array}{l}
\dot{x}=-y- \dfrac{4}{3} \,x^2,\\[8pt]
\dot{y}=x-\dfrac{16}{3}\,xy.
\end{array}
\right.
\end{equation}
Here, the dot denotes the derivative with respect to time. The phase portraits of these systems are depicted in Figure~\ref{fi:S1S2S3}.
\begin{figure}[h]
	\begin{center}
		\includegraphics[height=3cm]{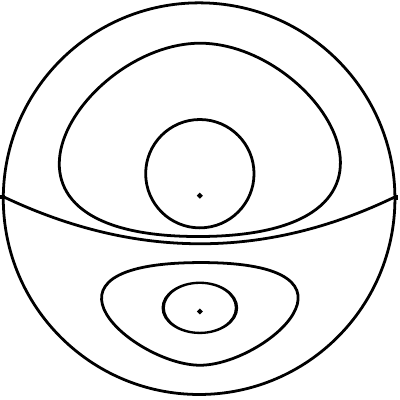} \quad
		\includegraphics[height=3cm]{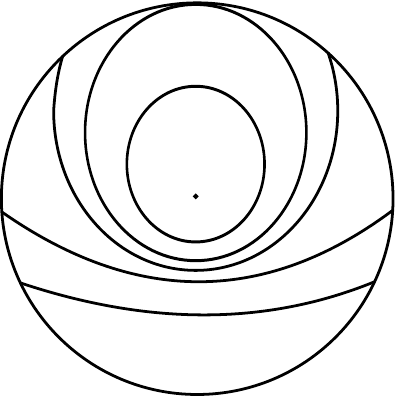} \quad
		\includegraphics[height=3cm]{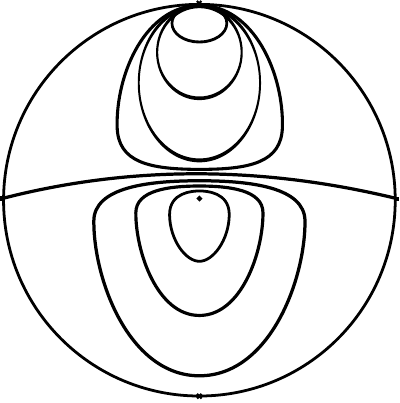} \quad
		\caption{Phase portrait of systems $S_1,$ $S_2,$ and $S_3$ from left to right.} \label{fi:S1S2S3}
	\end{center}
\end{figure}

The first order averaging method was used in \cite{MereuLlibre2014} to get 4 and 5 limit cycles by perturbing, respectively, the centers $S_1$ and $S_2$ inside the class of piecewise quadratic systems with two zones separated by the straight line $y=0.$ Here, due to restrictions of the employed technique, we take $\{x=0\}$ as the curve of discontinuity for the centers $S_1$ and $S_3.$ The first and second order analysis for $S_1$ are performed in Propositions~\ref{pr:S1O1} and \ref{pr:S1O2}, where we get 5 and 11 limit cycles, respectively. Analogously, for $S_3$ Propositions~\ref{pr:S3O1} and \ref{pr:S3O2} provide 5 and 10 limit cycles, respectively. We shall see that for the center $S_2$ the employed technique works whenever the curve of discontinuity is a straight line passing through the origin. This allows to reach the best result, namely $H^c_p(2)\geq16.$ In fact, proceeding with a first order analysis Proposition~\ref{pr:S2O1} provides $5,$ $6,$ and $8$ limit cycles when the curve of discontinuity is $\{x=0\},$ $\{y=0\},$ and $\{y+\sqrt{3}x=0\},$ respectively. Due to the difficulties in the massive computations, the second order analysis has been performed only for the case of highest cyclicity at the first order analysis, namely when the curve of discontinuity is given by $\{y+\sqrt{3}x=0\}.$ In this case, Proposition~\ref{pr:S2O2} provides $16$ limit cycles bifurcating from the origin, that is $M^c_p(2)\geq16.$ This proves our main result, Theorem~\ref{thm:Main}.

This work is structured as follows. In Section~\ref{se:2}, we present some basic notions and preliminary tools. In Sections~\ref{se:3} and \ref{se:4}, the first and second order analysis are performed, respectively. Finally, last section contains the integrals used in this paper.

\section{Preliminaries} \label{se:2}

This section is devoted to present some basic notions and preliminary tools needed to prove our main result. Firstly, we introduce some results on averaging theory of first and second orders. In fact, the limit cycles will appear from the simple zeros of some integrals (see, for instance, \cite{ILN,LlbAnaNovaes2015}). Secondly, we recall the concepts of Extended Complete Chebyshev system (ECT-system) and Chebyshev system with accuracy (see, for instance, \cite{NovTor2017}). Then, we introduce the concept of pseudo-Hopf bifurcation, which is the birth of a limit cycle when the sliding set changes stability (see, for instance, \cite{Fil1988,PonFreTor2012}). Finally, we state the Poincar\'e--Miranda theory, which is an extension of the intermediate value theorem, see \cite{Kul1997}.

\subsection{Averaging Theory}
Assume that the origin is a center equilibrium point for system \eqref{eq:0}. Consider the following perturbed piecewise polynomial vector field
\begin{equation}\label{eq:8}
 Z_\e^{\pm}=\left\{
\begin{array}{l}
 \!\!Z^+(x,y) + \e \,(P_1^+(x,y),Q_1^+(x,y))+ \e^2 \,(P_2^+(x,y),Q_2^+(x,y)), \textrm{ if } h(x,y)>0,\\
 \!\!Z^-(x,y) + \e \,(P_1^-(x,y),Q_1^-(x,y))+ \e^2 \,(P_2^-(x,y),Q_2^-(x,y)), \textrm{ if } h(x,y)<0,
\end{array}
\right. 
\end{equation}
where $\e$ is sufficiently small, $P_k^{\pm},$ $Q_k^{\pm}$ are polynomials of degree $n$ in $(x,y),$ for $k=1,2,$ and $h(x,y)=y-\tan(\al) x.$ After changing to polar coordinates, $(x,y)=(r\cos\T,r\sin\T),$ system \eqref{eq:8} writes
\begin{equation}\label{eq:9}
\widetilde Z_{\e}(\T,r)=\left\{
\begin{array}{ccl}
\widetilde Z^+_\e(\T,r),\quad \textrm{if}\quad \alpha<\T<\alpha+\pi,\\
\widetilde Z^-_\e(\T,r), \quad \textrm{if}\quad \alpha-\pi<\T<\alpha.
\end{array}
\right. 
\end{equation}
Taking $\T$ as the new independent variable, the differential system associated to the vector field \eqref{eq:9} becomes the piecewise differential equation
\begin{equation}\label{eq:drdt}
r'(\T)=\frac{dr}{d\theta}=\e F_1(\T,r)+\e^2 F_2(\T,r)+\CO(\e^3),
\end{equation}
with
\[
F_i(\T,r)=\left\{
\begin{array}{ccc}
F_i^+(\T,r)&\textrm{if}&\alpha<\T<\alpha+\pi,\\
F_i^-(\T,r)&\textrm{if}&\alpha-\pi<\T<\alpha,
\end{array}\right.
\]
where $F_i^{\pm}:[\alpha-\pi,\alpha+\pi]\times (0,\rho^*)\rightarrow\R$ are analytical functions $2\pi$--periodic in the variable $\T$ for $i=1,2.$
\smallskip

We define $\mathcal{F}_1,\mathcal{F}_2:(0,\rho^*)\rightarrow\R$ as
\begin{equation}\label{eq:F1F2}
\begin{array}{RL}
\mathcal{F}_1(r)=&\int_{\alpha}^{\alpha+\pi}\left(F_1^+(\T,r)+F_1^-(\T-\pi,r)\right)d\T,\\[10pt]
\mathcal{F}_2(r)=&\int_{\alpha}^{\alpha+\pi}\left(F_2^+(\T,r)+F_2^-(\T-\pi,r)\right)d\T\\
&+\int_{\alpha}^{\alpha+\pi}\left(\dfrac{\p}{\p r}F_1^+(\T,r)r_1^+(\T,r)+\dfrac{\p}{\p r}F_1^-(\T-\pi,r)r_1^-(\T-\pi,r)\right)d\T.
\end{array}
\end{equation}
Here, the functions $r_1^{\pm}:(-\pi,\pi)\times \R^+\rightarrow \R$ are defined as
\begin{equation}\label{r1}
r_1^{\pm}(\T,r)=\int_{\alpha}^{\alpha+\T}F_1^{\pm}(\phi,r)d\phi.
\end{equation}
\begin{thm}[\cite{LlbAnaNovaes2015}]\label{thm:averagingF1F2} Consider the piecewise differential equation \eqref{eq:drdt}.
\begin{enumerate}[(i)]
\item Suppose that for $\rho\in (0,\rho^*)$ with $\mathcal{F}_1(\rho)=0$ and $\mathcal{F}_1'(\rho)\neq 0.$ Then, for $|\e|>0$ sufficiently small, there exists a $2\pi$--periodic
solution $r(\T,\e)$ of \eqref{eq:drdt} such that $r(0,\e)\to \rho$ when $\e\to 0.$
\item Assume that $\mathcal{F}_1=0.$ Suppose that for $\rho\in (0,\rho^*)$ with $\mathcal{F}_2(\rho)=0$ and $\mathcal{F}_2'(\rho)\neq 0.$ Then, for $|\e|>0$ sufficiently small, there exists a $2\pi$--periodic solution $r(\T,\e)$ of \eqref{eq:drdt} such that $r(0,\e)\to \rho$ when $\e\to 0.$
\end{enumerate}
\end{thm}

\subsection{ECT-Systems}
Let $\CF=[u_0,\ldots, u_n]$ be an ordered set of functions of class $\CC^{\infty}$ on the closed interval $[a,b].$ We denote by $\mathcal Z(\CF)$ the maximum number of zeros counting multiplicity that any nontrivial function $v\in\Span(\CF)$ can have. Here, $\Span(\CF)$ is the set of functions generated by linear combinations of elements of $\CF,$ that is $v(s)=a_0 u_0(s)+a_1 u_1(s)+\cdots+a_n u_n(s)$ where $a_i,$ for $i=0,1,\ldots,n,$ are real numbers.

The theory of Chebyshev systems is a classical tool to study the quantity $\mathcal Z(\CF).$ In fact, when $\mathcal Z(\CF)\leq n,$ $\CF$ is called an {\it Extended Chebyshev system} or ET-system on $[a,b],$ see \cite{KarStu1966}. We say that $\CF$ is an {\it Extended Complete Chebyshev} system or an ECT-system on a closed interval $[a,b]$ if and only if for any $k,$ $0\leq k\leq n,$ $[u_0,u_1,\ldots,u_k]$ is an ET-system. In order to prove that $\CF$ is a ECT-system on $[a,b]$ it is sufficient and necessary to show that $W(u_0,u_1,\ldots,u_k)(t)\neq 0$ on $[a,b]$ for $0\leq k\leq n,$ see also \cite{KarStu1966}. Here, $W(u_0,u_1,\ldots,u_n)(t)$ denotes the Wronskian of $\mathcal F$ with respect to $t.$ That is, 
\begin{equation*}
W_n(t)=W(u_0,\ldots,u_n)(t)=\det\left(\begin{array}{ccc}
u_0(t)&\cdots&u_n(t)\\
u_0'(t)&\cdots&u_n'(t)\\
\vdots&\ddots&\vdots\\
u_0^{(n)}(t)&\cdots& u_n^{(n)}(t)
\end{array}\right).
\end{equation*}
Furthermore, the sufficient condition to be an ECT-system also provides that each configuration of $m\le n$ zeros, taking into account their multiplicity, is realizable.

The next theorem, proved in \cite{NovTor2017}, extends the results for ECT-systems when some of the Wronskian vanish.
\begin{thm}[\cite{NovTor2017}]\label{thm:NovTor2017}
Let $\CF=[u_0,u_1,\ldots,u_n]$ be an ordered set of analytic functions on $[a, b].$ Assume that all the $\nu_i$ zeros of the Wronskian $W_i$ are simple for $i = 0,\ldots, n.$ Then, the number of isolated zeros for every element of $\Span(\CF)$ does not exceed
\[
n+\nu_n+\nu_{n-1}+2(\nu_{n-2}+\cdots+\nu_0)+\nu_{n-1}+\cdots+\nu_3
\]
where $\nu_i=\min(2\nu_i,\nu_{i-3}+\cdots+\nu_0),$ for $i=3,\ldots,n-1.$
\end{thm}

\subsection{Pseudo-Hopf Bifurcation}
In the well-known Hopf bifurcation (see, for instance, \cite{HalKoc1991}) a limit cycle arises from an equilibrium point when it changes its stability. In piecewise differential systems, the pseudo-Hopf bifurcation describes the same phenomenon but when the sliding segment changes its stability. Analogously to the classical Hopf bifurcation, the proof is a direct consequence of the generalized Poincar\'e--Bendixson Theorem for piecewise differential systems (see, for instance, \cite{BuzCarEuz2018}).

\begin{prop}\label{pr:pseudoHopf}
Let $Z^{\pm}=(X^{\pm}(x,y),Y^{\pm}(x,y))$ be a $\CC^1$ piecewise differential system in two zones separated by the straight line $y=0.$ Additionally, the origin is a stable monodromic equilibrium point and $a\!=\!(\partial Y^{+}/\partial x)|_{(0,0)}\!>\!0.$ Given a real number $b,$ we consider the perturbed system $Z^{\pm}_{b}=(X^{\pm}_b(x,y),Y^{\pm}_b(x,y))$ defined by $X^{\pm}_b(x,y)=X^{\pm}(x,y),$ and $Y^{-}_b(x,y)=Y^{-}(x,y)$ and $Y_{b}^{+}=Y^{+}+b.$ Then, for $b$ small enough, the system $Z^{\pm}_{b}$ exhibits a pseudo-Hopf bifurcation at $b=0$ when $ab\!>\!0.$ See Figure~\ref{fi:pseudohopf}. 
\end{prop}
\begin{figure}[h]
\begin{center}
\begin{overpic}[height=1.5cm]{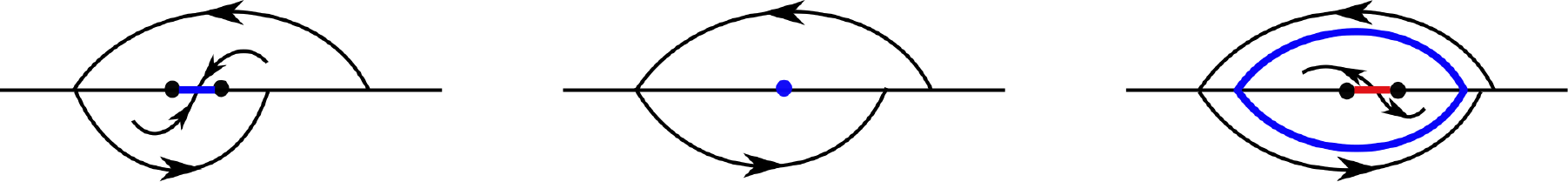}
\put(7,-3){$b<0$}
\put(45,-3){$b=0$}
\put(82,-3){$b>0$}
\end{overpic}
\end{center}
\caption{Pseudo-Hopf bifurcation.}\label{fi:pseudohopf}
\end{figure}

\subsection{Poincar\'e--Miranda Theorem}
The next result is a generalization of the intermediate value theorem. It was conjectured by Poincar\'e in 1883 and proved by Miranda in 1940 (see, for instance, \cite{Kul1997} and the references therein). 

\begin{thm}[\cite{Kul1997}]\label{thm:PoincareMiranda} Let $a$ be a positive real number and $B=[-a,a]^n$ the $n$-dimensional cube. Let $f=(f_1,\ldots,f_n)\!:\!B\rightarrow \R^{n}$ be a continuous function such that $f_i(B_i^-)<0$ and $f_i(B_i^+)>0,$ for each $i\leq n,$ where $B_i^\pm=\{(x_1,\ldots,x_{n})\in B: x_i=\pm a\}.$ Then, there exists a point $c\in B$ such that $f(c)=0.$
\end{thm}

\section{First order perturbation} \label{se:3}

In this section the first order averaging method is used to study the limit cycles of the perturbed piecewise vector field \eqref{eq:8} when the unperturbed vector field $Z_0$ is a quadratic isochronous center in one of the families $S_1,$ $S_2,$ or $S_3.$
Regarding \eqref{eq:8} we shall denote $Z_{\e}=Z_{i,\e}$ and $Z_{\e}^{\pm}=Z_{i,\e}^{\pm}$ in order to indicate that $Z_0\in S_i,$ for $i=1,2,3.$ Here, it is only considered quadratic polynomial perturbations, that is 
\[
 P_k^{\pm}(x,y)=\sum_{j=0}^{2}\sum_{i=0}^{j}p^{\pm}_{k,i,j-i}x^{i}y^{j-i} \quad\text{and}\quad Q_k^{\pm}(x,y)=\sum_{j=0}^{2}\sum_{i=0}^{j}q^{\pm}_{k,i,j-i}x^{i}y^{j-i}.
\]
The first order analyses for families $S_1,$ $S_3,$ and $S_2$ are performed in Propositions~\ref{pr:S1O1}, \ref{pr:S3O1}, and \ref{pr:S2O1}, respectively. For the families $S_1$ and $S_2,$ we shall also use the ECT-system properties to study the bifurcation of limit cycles in the global interval of definition. Accordingly, Propositions~\ref{pr:S1O1} and \ref{pr:S3O1} are concerned about upper bounds (up to a first order analysis) for the maximum number of limit cycles bifurcating from the period annulus (the so-called medium amplitude limit cycles). In the third result, Proposition~\ref{pr:S2O1}, a local analysis is performed around the center point. In this case, we also see how the number of limit cycles changes when we consider different lines of discontinuity. 

Before stating the main results of this section we briefly discuss the choosing of the lines of discontinuity. The birational linearizations of families $S_1$ and $S_3$ (see, for instance, \cite{ChaSab1999}) transform the straight line $\{x=0\}$ into another straight line passing through the origin. Moreover, $\{x=0\}$ is the unique straight line for which this happens. This is the main reason for choosing $\Sigma=\{x=0\}$ as the curve of discontinuity. The birational linearization of the family $S_2$ transforms straight lines passing through the origin into straight lines passing through the origin, so that we are allowed to choose any straight line passing through the origin as the curve of discontinuity. Nevertheless, in this last case, since the computations are more intricate we only study the limit cycles bifurcating from the origin. We anticipate that all the conclusions of this section will be improved by results of the next section.

\begin{prop}\label{pr:S1O1}
For $|\e|>0$ sufficiently small the averaging method of first order predicts at most $5$ crossing limit cycles for the piecewise quadratic vector field $Z_{1,\e}$ when the curve of discontinuity is the straight line $\{x=0\}.$ Moreover, this number is reached.
\end{prop}

\begin{proof} In order to apply Theorem~\ref{thm:averagingF1F2}, we have to write the vector field \eqref{eq:9} as a differential equation \eqref{eq:drdt}. So, we first proceed with the change of variables (see, for instance, \cite{ChaSab1999})
\begin{equation}\label{ls1}
x=-\dfrac{v}{v^2+(u-1)^2}\quad \textrm{and}\quad y=-\dfrac{u^2+v^2-u}{v^2+(u-1)^2},
\end{equation}
which has the following rational inverse
\begin{equation}\label{ls1inverse}
u=\dfrac{x^2+y^2+y}{x^2+y^2+2 y+1}\quad \textrm{and}\quad v=-\dfrac{x}{x^2+y^2+2y+1}.
\end{equation}
With this change of variables the differential equation $S_1$ becomes the linear center $(u',v')=(-v,u)$ and the line of discontinuity becomes $v=0.$ 

Then, we change to polar coordinates $u= r \cos\T$ and $v=r \sin \T.$ Taking $\T$ as the new independent variable, \eqref{eq:9} becomes
\begin{equation}\label{trans1}
r'(\T)=\dfrac{\dot r}{\dot{\T}}=\e\dfrac{\A(r \cos \T,r\sin\T)}{\CC(\T,r)}+\CO(\e^2),
\end{equation}
where $\CC(\T,r)= (2 r \cos \T-r^2-1)^2$ and $\A$ is the piecewise function
\begin{equation}\label{calA}
\A(r \cos \T,r\sin\T)=\left\{\begin{array}{l} \A^+(r \cos \T,r\sin\T) \quad \textrm{if} \quad 0<\T\leq \pi,\\
\A^-(r \cos \T,r\sin\T) \quad \textrm{if} \quad \pi<\T\leq 2\pi,\end{array}\right.
\end{equation}
being $\A^{\pm}$ polynomials of degree $3.$

From here we want to use the integral formulas of Section~\ref{se:5} to compute the averaged function $\CF_1,$ as stated in \eqref{eq:F1F2}, for $\al=0.$ The denominators of $F_1^+(\T,r)$ and $F_1^-(\T-\pi,r)$ write $(2r\cos \theta-r^2-1)^2$ and $(2r\cos \theta+r^2+1)^2,$ respectively. In order to use the integrals we must apply a transformation on $r$ in order to get the denominators written in a standard form. 

Firstly, take $r=(-1+\sqrt{1-R^2})/R.$ The denominator of $F_1^+(\T,r)$ is transformed into $2 R^2(R\cos \theta+1)^2(R^2+2\sqrt{1-R^2}-2).$
Hence, the first part of the first averaged function
\begin{equation*}\label{int1}
\int_{0}^{\pi}F_1^+(\T,(-1+\sqrt{1-R^2})/R)d\T
\end{equation*}
can be computed using the integrals \eqref{eq:11} for $\alpha=0,$ $\ell=2,$ and $k\in\{0,1,2,3\}.$ We shall suppress it here. The original variable $r$ is recovered by taking $R=-2r/(r^2+1).$

Secondly, take $r=(1-\sqrt{1-R^2})/R.$ The denominator of $F_1^-(\T-\pi,r)$ is transformed into $2 R^2(R\cos \theta+1)^2(R^2+2\sqrt{1-R^2}-2).$ Hence, the second part of the first averaging function
\begin{equation*}\label{int2}
\int_{0}^{\pi}F_1^-(\T-\pi,(1-\sqrt{1-R^2})/R)d\T
\end{equation*}
can be computed also using the integrals \eqref{eq:11} for $\alpha=0,$ $\ell=2,$ and $k\in\{0,1,2,3\}.$ The original variable $r$ is recovered by taking $R=2r/(r^2+1).$

Adding up the above integrals, we get the averaged function $\CF_1(r).$ Proceeding with the change of parameters
\begin{equation}\label{eq:pqs1}
\begin{aligned}
p^-_{1,0,0}&=p^+_{1,0,0}-2 k_4+\dfrac{k_0}{2},\\
p^-_{1,0,1}&=p^-_{1,0,2}+p^+_{1,0,1}-p^+_{1,0,2}-2 k_4+2 k_5+\dfrac{k_2+k_0}{2},\\
p^-_{1,1,0}&=-p^+_{1,1,0}+2( q^-_{1,0,0}+ q^+_{1,0,0})- \dfrac{ q^-_{1,0,2}+ q^-_{1,2,0} + q^+_{1,0,2} + q^+_{1,2,0}}{2} + \dfrac{k_1-k_3}{\pi},\\
p^-_{1,2,0}&=p^+_{1,2,0}+\dfrac{q^-_{1,1,1}- q^+_{1,1,1}}{2}-2 k_4-2 k_5,\\
q^-_{1,0,1}&=2q^-_{1,0,0}+2 q^+_{1,0,0}- q^+_{1,0,1}+\dfrac{q^-_{1,0,2}+ q^-_{1,2,0}+ q^+_{1,0,2}+ q^+_{1,2,0}}{2}+ \dfrac{k_1+2 k_3}{\pi} ,\\
q^+_{1,1,0}&=q^-_{1,1,0}+\dfrac{q^+_{1,1,1}-q^-_{1,1,1}+k_2+k_0}{2},\\
\end{aligned}
\end{equation}
we get
\begin{equation}\label{mels1}
\CF_1(r)= \sum_{n=0}^5 k_n f_n(r),
\end{equation}
where
\[
\begin{array}{l}
f_0(r)=1,\quad f_1(r)=r,\quad f_2(r)=r^2, \quad
f_3(r)=r^3,\vspace{0.2cm}\\
f_4(r)=\dfrac{1-r^2}{r}L(r), \quad
f_5(r)=r(1-r^2)\,L(r),
\end{array}
\]
and
\begin{equation}\label{log}
L(r)=\log\left(\frac{1-r}{1+r}\right). 
\end{equation}
Clearly, from \eqref{eq:pqs1}, the parameters $k_n$ can be chosen arbitrarily.

The maximum number of simple zeros that \eqref{mels1} can have follows by studying the Wronskians of the ordered set $[f_0,f_1,\ldots,f_5].$ Straightforward computations show that
\[
\begin{array}{l}
W_0(r)=1, \quad W_1(r)=1, \quad W_2(r)=2,\quad W_3(r)=12,\vspace{0.2cm}\\
W_4(r)=\dfrac{288}{r^5} \, \overline{W}_4(r), \quad W_5(r)=\dfrac{9216(r^2+5)}{(1-r^2)^4r^6}\,\overline{W}_5(r), 
\end{array}
\]
where
\[
\overline{W}_4=L(r)-\dfrac{2}{3}\,\dfrac {r \left( 7r^2-8r^2+3 \right) }{(r^2-1)^3}, \qquad
\overline{W}_5=L(r)-\dfrac{2}{3}\,\dfrac {r \left( 3r^2-22r^2+15 \right) }{(r^2-1)^2(r^2+5)}.
\]
Clearly, $W_0, W_1, W_2,$ and $W_3$ do not vanish in $(0,1).$ Now, computing the derivative 
\[
\overline{W}_4'(r)=\dfrac{4 r^4(5r^2+1)}{(r^2-1)^4}>0,
\]
as $\overline{W}_4(0)=0,$ also $W_4(r)$ is does not vanish in $(0,1).$ The same argument applies for $W_5(r),$ but using 
\[
\overline{W}_5'(r)=\dfrac{64r^6}{(r^2-1)^3(r^2+5)^2}.
\]
So, the proof follows by noticing that the ordered set of functions $[f_0,\ldots,f_5]$ is an ECT-system.
\end{proof}

\begin{prop}\label{pr:S3O1}
For $|\e|>0$ sufficiently small the averaging method of first order predicts at most $5$ crossing limit cycles for the piecewise quadratic vector field $Z_{3,\e}$ when the curve of discontinuity is the straight line $\{x=0\}.$ Moreover, this number is reached
.\end{prop}

\begin{proof} We shall follow the same procedure of the proof of Proposition~\ref{pr:S1O1}. The linearization stated in \cite{ChaSab1999} is given by 
\begin{equation*}\label{2s1}
x=\dfrac{3 u}{8 v +1}\quad y=\dfrac{3(4 u^2+8 v^2+v)}{(8 v+1)^2},
\end{equation*}
which has the following rational inverse 
\begin{equation*}\label{2s1inverse}
u=\dfrac{3x}{32x^2-24y+9}\quad \textrm{and}\quad v=\dfrac{-4x^2+3y}{32x^2-24y+9}.
\end{equation*}
Then, applying the change of variables $u= r \sin\T$ and $v=-r \cos \T,$ and taking $\T$ as the new independent variable, equation \eqref{eq:drdt} becomes
\begin{equation*}
r'(\T)=\dfrac{\dot r}{\dot{\T}}=\e\dfrac{\A(\T,r)}{\CC(\T,r)}+\CO(\e^2),
\end{equation*}
where $\CC(\T,r)= 8r\cos \theta-1$ and $\A$ is the piecewise function
\begin{equation*}
\A(\T,r)=\left\{\begin{array}{l} \A^+(r \sin \T,-r\cos\T) \quad \textrm{if} \quad 0<\T\leq \pi,\\
\A^-(r \sin \T,-r\cos\T) \quad \textrm{if} \quad \pi<\T\leq 2\pi,\end{array}\right.
\end{equation*}
being $\A^{\pm}$ polynomials of degree $6.$ 

Now we compute the averaged function \eqref{eq:F1F2} for $\alpha=0.$ As in the proof of Proposition~\ref{pr:S1O1}, the denominators of $F_1^+(\T,r)$ and $F_1^-(\T-\pi,r)$ are not written in a standard form in order to use directly the integrals of Section~\ref{se:5}. 

Firstly, take $r=-R/8.$ The denominator of $F_1^+(\T,r)$ in \eqref{eq:F1F2} becomes $(R\cos \theta+1)^4.$ Hence, the integral
\[
\int_{0}^{\pi}F_1^+(\T,-R/8)d\T
\]
can be computed using \eqref{eq:11} for $\alpha=0,$ $\ell=4,$ and $k\in\{0,1,\ldots,6\}.$ The original variable $r$ is recovered taking $R=-8r.$

Secondly, take $r=R/8.$ The denominator of $F_1^-(\T-\pi,r)$ in \eqref{eq:F1F2} also becomes $(R\cos \theta+1)^4.$ Hence, the integral
\[
\int_{0}^{\pi}F_1^-(\T-\pi,R/8)d\T,
\]
can be also computed using \eqref{eq:11} for $\alpha=0,$ $\ell=4,$ and $k\in\{0,1,\ldots,6\}.$ The original variable $r$ is recovered taking $R=8r.$

Adding up the above integrals we obtain the first averaged function $\CF_1(r),$ which depends on $r,$ $\sqrt{1-r^2},$ and $L(r)$ defined in \eqref{log}. Proceeding with the change 
\begin{equation}\label{change}
r=2\rho/(1+\rho^2),
\end{equation}
the averaged function writes
\begin{equation}\label{mels3}
\widetilde\CF_1(\rho)= \sum_{n=0}^5 k_n f_n(\rho),
\end{equation}
where
\[
\begin{array}{l}
f_0(\rho)=\rho,\quad f_1(\rho)=\rho^2, \quad f_2(\rho)=\rho^3,\vspace{0.2cm}\\
f_3(\rho)=\rho^4+1, \quad f_4(\rho)=\rho^5, \quad f_5(\rho)=L(\rho)/\rho. \vspace{0.2cm}\\
\end{array}
\]
We remark that $L(r)=2L(\rho).$

The maximum number of simple zeros that \eqref{mels3} can have follows by studying the Wronskians of the ordered set $[f_0,f_1,\ldots,f_5].$ Straightforward computations show that.
\[
\begin{array}{l}
W_0(\rho)=\rho, \quad W_1(\rho)=\rho^2, \quad W_2(\rho)=2\rho^3 ,\quad W_3(\rho)=12(\rho^4-1),\vspace{0.2cm}\\
W_4(\rho)=288\rho(\rho^4-5), \quad W_5(\rho)=\dfrac{207360(1-\rho^4)}{\rho^5}\,\ov W_5(\rho),
\end{array}
\]
where
\[
\ov W_5(\rho)=L(\rho)-\,{\frac {\rho\, \left( 75\,{\rho}^{12}-175\,{\rho}^{10}+61\,{\rho}^{
8}+95\,{\rho}^{6}-230\,{\rho}^{4}+140\,{\rho}^{2}-30 \right) }{15 \left( 
\rho^2-1 \right)^{6}\left( {\rho}^{2}+1
 \right) }}.
\]
Clearly $W_0, W_1, W_2 , W_3,$ and $W_4$ do not vanish in $(0,1).$ The last Wronskian does not vanish either because $\ov W_5(0)=0$ and the derivative 
\[
\ov W_5'(\rho)\!=\!\,{\frac {{\rho}^{4} (\rho^4-5)(105\rho^8+105\rho^6+175\rho^4-5\rho^2+4) }{
 15\left( \rho^2\!-\!1 \right) ^{7} \left( {\rho}^{2}\!+\!
1 \right) ^{2}}}
\]
is positive for every $\rho\in(0,1).$ So, the proof follows by noticing that the ordered set of functions $[f_0,\ldots,f_5]$ is an ECT-system.
\end{proof}

The global analysis performed in the previous results cannot be performed in a straightforward way for the family $S_2.$ Hence, for this family we provide only a local analysis around the origin. 

\begin{prop}\label{pr:S2O1}
For $|\e|>0$ sufficiently small and under the condition $P^{\pm}(0,0)=Q^{\pm}(0,0)=0,$ the averaging method of first order predicts at most $4, 5,$ or $7$ limit cycles bifurcating from the origin for the quadratic vector field $Z_{2,\e}$ when the curve of discontinuity is the straight line $\{x=0\},$ $\{y=0\},$ or $\{y+\sqrt{3}x=0\},$ respectively. Moreover, these numbers are reached.
\end{prop}

\begin{proof} 
The linearization stated in \cite{ChaSab1999} for family $S_2$ is given by
\begin{equation*}\label{ls2}
x=-\dfrac{u}{v-1}\quad \textrm{and}\quad y=-\dfrac{v}{v-1},
\end{equation*}
which has the following rational inverse 
\begin{equation*}\label{invls2}
u=\dfrac{x}{y+1}\quad \textrm{and}\quad v=\dfrac{y}{y+1}.
\end{equation*}
As we have commented before, straight lines passing through the origin are transformed into straight lines passing through the origin.	
	
Firstly, assume that $\Sigma=\{x=0\}.$ Applying the change of variable $(u,v)=(r\sin \theta,-r\cos \theta)$ and taking $\theta$ as the new independent variable we obtain the equivalent functions \eqref{trans1} and \eqref{calA}. Here, $\A^{\pm}$ are cubic polynomials and the denominator becomes $\CC(\T,r)=1+r \cos \T.$ For this case, the first averaged function $\CF_1$ is given by $\eqref{eq:F1F2}$ for $\alpha=0.$ Since $\CC(\T,r)=1+r\cos \theta$ is the denominator of $F^{+}_1(\T,r)$ in \eqref{eq:F1F2}, the integrals \eqref{eq:11} can be used directly. Nevertheless, the denominator of $F^{-}_1(\T-\pi,r)$ in \eqref{eq:F1F2} is given by $\CC(\T-\pi,r)=1-r\cos \theta,$ so it is necessary to proceed with the change $r=-R$ in order to use the integrals \eqref{eq:11}. Applying the integrals \eqref{eq:11} for $\ell=1$ and $k\in\{0,1,2,3\},$ and going back to the original variable $r$ we have computed the first averaged function $\CF_1(r).$ Finally, with the change \eqref{change} and after some algebraic manipulations, we get
\[
\widetilde{\CF}_1(\rho)= \sum_{n=0}^4 k_n f_n(\rho),
\]
with
\begin{equation*}
\begin{aligned}
&f_0(\rho)={\frac {\rho}{ \left( {\rho}^{2}+1 \right) ^{2}}}, \qquad
f_1(\rho)={\frac {{\rho}^{2}}{ \left( {\rho}^{2}+1 \right) ^{2}}},\\
&f_2(\rho)={\frac {3\,{\rho}^{4
}+3\,{\rho}^{3}+{\rho}^{2}+3}{3 \left( {\rho}^{2}+1 \right) ^{2}}}+\frac {{\rho}^{4}-{\rho}^{2}+1}{ 2\left( {\rho}^{2}+1 \right) \rho}
L(\rho),\\
&f_3(\rho)=-{\frac {3\rho^2}{ 4\left( {\rho}^{2}+1 \right) ^{2}} }-\frac {3\rho}{ 8\left( {\rho}^{2}+1 \right) } L(\rho), \quad
f_4(\rho)={\frac {{\rho}^{5}}{ \left( {\rho}^{2}+1 \right) ^{2}}},
\end{aligned}
\end{equation*}
and $L$ is defined in \eqref{log}. Moreover, the parameters $k_n$ can be chosen arbitrarily. The first part of the statement follows because, in a neighborhood of the origin, $f_i(\rho)=\rho^{i+1}+O(\rho^{i+2}).$

Now, assume that $\Sigma=\{y=0\}.$ The procedure for this case is similar to the previous case. We only detail the differences. The functions $\CF_1$ and $\widetilde{\CF}_1$ are obtained from \eqref{eq:F1F2} and \eqref{eq:11}, but now for $\alpha=-\pi/2.$ Thus, after some algebraic manipulations we get
\[
\widetilde{\CF}_1(\rho)= \sum_{n=0}^5 k_n f_n(\rho),
\]
with
\begin{equation}\label{eqrho}
\begin{aligned}
f_0(\rho)=& \dfrac{\rho}{(\rho^2 + 1)^2},\quad f_1(\rho) = \dfrac{\rho^2}{(\rho^2 + 1)^2},\quad f_2(\rho) = \dfrac{\rho^3}{(\rho^2 + 1)^2},\vspace{0.2cm}\\ 
f_3(\rho) =&\frac{3\rho^2}{4(\rho^2 + 1)^2} -\dfrac{3\rho(\rho^2-1)^2}{8(\rho^2 + 1)^3}\phi\left(\rho,\frac{\pi}{2}\right),\quad f_4(\rho) = \dfrac{\rho^5}{(\rho^2 + 1)^2},\\
f_5(\rho) =&{\frac {525\,{\rho}^{4}-490\,{\rho}^{2}+525}{768\,
		\left( {\rho}^{2}+1 \right) ^{2}}}-{\frac { \left( 175\,{\rho}^{4}+70\,{\rho}^{2}+175 \right) \left( \rho^2-1 \right) ^{2}}{512\,\rho\, \left( {\rho}^
		{2}+1 \right) ^{3}}} \phi\left(\rho,\frac{\pi}{2}\right).
\end{aligned}
\end{equation}
Here, the function $\phi$ is defined as
\begin{equation*}
\phi(r,\theta)=\dfrac{1}{\sqrt{1-r^2}}\left(\theta -2 \arctan\left(\sqrt{\dfrac{1-r}{1+r}}\tan\left(\dfrac{\theta}{2}\right) \right)\right),
\end{equation*}
and the parameters $k_n,$ $n=0,1,\ldots,5,$ are arbitrary real numbers. The functions \eqref{eqrho} also write $f_i(\rho)=\rho^{i+1}+O(\rho^{i+2}).$ Consequently, the second part of the proof follows.

Finally, assume that $\Sigma=\{y+\sqrt{3}x=0\}.$ Again, the procedure for this case is similar to the previous cases and we shall only detail the differences. 
The functions $\CF_1$ and $\widetilde{\CF}_1$ are obtained from \eqref{eq:F1F2} and \eqref{eq:11}, but now for $\alpha=-\pi/3.$ Thus, after some algebraic manipulations we get
\[
\widetilde{\CF}_1(\rho)= \sum_{n=0}^7 k_n f_n, 
\]
with
\begin{equation}\label{eqmu}
\begin{aligned}
f_0(\rho) = &\frac{\rho}{\rho^2+1},\quad f_1(\rho) = \frac{\rho^2}{(\rho^2+1)^2},\quad f_2(\rho) = \frac{\rho^3}{(\rho^2+1)^2}, \quad f_4(\rho) = \frac{\rho^5}{(\rho^2+1)^2},\\
f_3(\rho) = &\frac {5(54733 \rho^4+94452 \rho^2+54733)}{6912 (\rho^2+1)^2}+\frac{15( 1366 \rho^4+1847 \rho^2+1366 )}{1024(\rho^2+1) \rho}\,\widetilde{L}(\rho)\\
	&+\frac {25 \sqrt{3}(236 \rho^4-247 \rho^2+236)(\rho^2-1)^2}{82944 \rho(\rho^2+1)^3}\,\widetilde{\phi}(\rho),\\
f_5(\rho) = &-{\frac {35(21835 \rho^4+40596 \rho^{2}+21835)}{6912 (\rho^{2}+1)^{2}}}-{\frac { 105( 550 \rho^4+797 \rho^{2}+550)}{1024(\rho^{2}+1)\rho}}\,\widetilde{L}(\rho)\\
	&-{\frac {175 \sqrt {3}(176 \rho^4-181 \rho^{2}+176) (\rho^2-1)^2}{82944 \rho ( \rho^{2}+1 ) ^3}}\,\widetilde{\phi}(\rho),\\
f_6(\rho) = &{\frac {245(227\rho^4+444 \rho^{2}+227)}{768 (\rho^{2}+1)^2}}+{\frac {315 ( 122 \rho^4+181 \rho^{2}+122 ) }{1024(\rho^{2}+1)\rho}}\,\widetilde{L}(\rho)\\
	&+{\frac {35 \sqrt {3}(116 \rho^4-115 \rho^{2}+116) (\rho^2-1)^2}{9216 \rho (\rho^2+1)^3}}\,\widetilde{\phi}(\rho),\\
f_7(\rho) = &-{\frac {385(77 \rho^4+156 \rho^{2}+77)}{2304 (\rho^{2}+1)^2}}-{\frac { 3465 ( 2 \rho^4+3 \rho^{2}+2 )}{1024(\rho^{2}+1)\rho}}\,\widetilde{L}(\rho)\\
	&-{\frac {385 \sqrt {3}(8\rho^4-7 \rho^{2}+8 ) (\rho^2-1)^2}{27648 \rho(\rho^{2}+1)^3}}\,\widetilde{\phi}(\rho).\\
\end{aligned}
\end{equation}
Here,
\[
\widetilde{L}(\rho)=\log \left( {\frac {\rho^2-{\rho}+1}{\rho^2+{\rho}+1}} \right),
\quad 
\widetilde{\phi}(\rho)=\phi\left(\dfrac{2\rho}{\rho^2+1},\dfrac{2\pi}{3}\right)-\phi\left(-\dfrac{2\rho}{\rho^2+1},\dfrac{2\pi}{3}\right).
\]
Analogously to the previous cases, $k_n,$ $n=0,1,\ldots,7,$ are arbitrary real numbers. Here, $f_i(\rho)=\rho^{i+1}+O(\rho^{i+2})$ for $i=0,1,\ldots,5,$ $f_{6}(\rho)=\rho^{8}+O(\rho^{9}),$ and $f_{7}(\rho)=\rho^{10}+O(\rho^{11}).$ Therefore, the ordered set of functions $[f_0,f_1,\ldots,f_7]$ is an ECT-system in a neighborhood of the origin. This completes the proof of the last case.
\end{proof}

Following the ideas of \cite{CruTor2018}, the previous local result can be numerically improved to a global one. Taking linear combinations of the functions \eqref{eqmu} one may try to get an ordered set of $8$ functions which is an ECT-system with accuracy (see, for instance, \cite{NovTor2017}).
For instance, it can be checked numerically that the ordered set $[f_0,f_1,f_2,f_3,f_4,f_6+f_7,f_7,f_5]$ has all Wronskians non-vanishing except $W_5, W_6$ which vanish once. From Theorem~\ref{thm:NovTor2017}, we conclude that $\CF$ has at most $9$ simple zeros. We shall see that a second order analysis allow us to overcome this number of limit cycles. 

The next result is a technical lemma describing the existence of a pseudo-Hopf bifurcation for $Z_{2,\e}.$ 

\begin{lem}\label{le:hopfzie}
Consider the piecewise vector fields $Z_{i,\e},$ $i=1,2,3,$ under the assumption
$P^{\pm}(0,0)=Q^{\pm}(0,0)=0.$ For all curve of discontinuity given by $\{h(x,y)=Ax+By=0\},$ there exists a constant perturbation such that a small limit cycle bifurcates from the origin in a pseudo-Hopf bifurcation. 
\end{lem}

\begin{proof}
The unperturbed vector fields have a monodromic equilibrium point. This property remains under the assumption $P^{\pm}(0,0)=Q^{\pm}(0,0)=0.$ Then, the proof follows directly from Proposition~\ref{pr:pseudoHopf}.
\end{proof}

The conclusions on hyperbolic limit cycles of $Z_{2,\e}$ provided by Proposition~\ref{pr:S2O1} have assumed that $P^{\pm}(0,0)=Q^{\pm}(0,0)=0.$ So, from Lemma~\ref{le:hopfzie}, the parameters $P^{\pm}(0,0)$ and $Q^{\pm}(0,0)$ can be used to get a pseudo-Hopf bifurcation for $Z_{2,\e}^{\pm},$ which adds an extra limit cycle to each case of Proposition~\ref{pr:S2O1}. This is the content of the next result. It is worthwhile to say that this is the best result so far obtained after a first order analysis for piecewise quadratic system in two zones separated by a straight line.

\begin{cor}\label{co:Z2e}
For $|\e|>0$ sufficiently small, the maximum number of limit cycles that the system $Z_{2,\e}$ can have in any neighborhood of origin is at least $5,$ $6,$ and $8$ when the curve of discontinuity is $\{x=0\},$ $\{y=0\},$ and $\{y+\sqrt{3}x=0\},$ respectively.
\end{cor}

\section{Second order perturbation}\label{se:4}

In this section, in order to extend the previous results, we perform a second order analysis on piecewise quadratic perturbations of quadratic isochronous centers from the families $S_1,$ $S_2,$ and $S_3$ (see \eqref{eq:10}). More specifically, we shall apply the averaging method of second order to study the limit cycles of $Z_{i,\e},$ $i=1,2,3.$ Due to the difficulties in the massive second order computations, we only perform a local study. Despite this, we shall get the best lower bounds so far for the maximum number of limit cycles of $Z_{i,\e},$ $i=1,2,3,$ which are $11,$ $16,$ and $10,$ respectively. This proves our main result, Theorem~\ref{thm:Main}.

In Propositions~\ref{pr:S1O2} and \ref{pr:S3O2}, we provide conditions such that the second averaged functions associated to $Z_{1,\e}$ and $Z_{3,\e}$ are linear with respect to the parameters and have the highest possible rank. Under these conditions the origin is a zero of maximal finite multiplicity for $\CF_2.$ Moreover, we shall see that $\CF_2$ satisfies the versal unfolding property at the origin guaranteeing then the existence of the highest possible number of simple zeros near the origin and, consequently, limit cycles for $Z_{1,\e}$ and $Z_{3,\e}.$ The second order analysis for centers of the family $S_2$ is much more difficult and the procedure used for the families $S_1$ and $S_3$ cannot be followed straightly for $S_2$. In this case, some computer assisted analyses will be needed in order to use the Poincar\'{e}--Miranda theorem, that is Theorem~\ref{thm:PoincareMiranda}, to obtain analytically the existence of 16 limit cycles of $Z_{2,\e}$ bifurcating from the origin. This is the content of Proposition~\ref{pr:S2O2}.

\begin{prop}\label{pr:S1O2}
For $|\e|>0$ sufficiently small, the maximum number of limit cycles that $Z_{1,\e}$ can have in any neighborhood of the origin is at least 11 when the curve of discontinuity is $\{x=0\}.$
\end{prop}

\begin{proof} 
Assume that $P^{\pm}(0,0)=Q^{\pm}(0,0)=0.$ Under such condition, as in Proposition~\ref{pr:pseudoHopf} or Lemma~\ref{le:hopfzie}, an extra limit cycle can always be obtained from a pseudo-Hopf bifurcation. So, the rest of the proof consists in applying the second order averaging method, Theorem~\ref{thm:averagingF1F2}(ii), to get at least 10 limit cycles bifurcating from the origin. We notice that such special condition on the perturbation will guarantee that the averaged functions $\CF_1$ and $\CF_2$ are well defined at the origin. 

The proof is structured in two parts. Firstly, we provide the expression of the function $\CF_2(r).$ Secondly, we study the Taylor series of $\CF_2$ around $r=0$ in order to obtain the highest number of independent monomials.

\medskip

The first part will follow the same steps as in the proof of Proposition~\ref{pr:S1O1}. In fact, the function $\CF_1$ is given by \eqref{mels1}. Then, imposing conditions such that $\CF_1\equiv 0,$ that is $k_0=\cdots=k_5=0$ in \eqref{eq:pqs1}, we compute the second averaged function $\CF_2$ from \eqref{eq:F1F2} for $\al=0$. Proceeding with the changes of variables \eqref{ls1} and \eqref{ls1inverse} the denominators of the functions $F^{\pm}_1$ and $F^{\pm}_2$ write $\left(1+R\cos\psi\right)^2.$ Hence, the integrals
\begin{equation*}\label{int3}
\begin{aligned}
r^{+}_1(\T,R)&=\int_{0}^{\T}F_1^+(\T,(-1+\sqrt{1-R^2})/R)d\T,\\
r^{-}_1(\T-\pi,R)&=\int_{0}^{\T}F^{-}_1(\T-\pi,(1-\sqrt{1-R^2})/R)d\T,
\end{aligned}
\end{equation*}
can be computed using the expressions $\mathcal{\{S,C\}}_{k,\ell}^{\alpha=0}$ (see \eqref{eq:11}) for $\ell=2,$ and $k\in\{0,1,2,3\}.$ We notice that $\CF_2=\CF_2^{(1)}+\CF_2^{(2)}$ in \eqref{eq:F1F2} has two summands. The first one, which depends linearly on the parameters of second order terms ($p^\pm_{2,i,j},q^\pm_{2,i,j}$ in $F_2^\pm$), has the same form as $\CF_1.$ Indeed, changing the first index 1 to 2 of all the parameters $p^\pm_{1,i,j},q^\pm_{1,i,j}$ in \eqref{eq:pqs1} we see that $\CF_2^{(1)}$ becomes $\CF_1.$ Consequently, $\CF_2^{(1)}$ writes as \eqref{mels1} for some new parameters $k_0,k_1,\ldots,k_5.$ The second summand, which depends quadratically on the remaining parameters of first order terms ($p^\pm_{1,i,j},q^\pm_{1,i,j}$ in $F_1$), can also be obtained using the integrals from Section~\ref{se:5}. Indeed, in order to get $\CF_2^{(2)}$ the integrals
\begin{equation*}\label{int5}
\begin{aligned}
\mathcal{G}^+(R)&=\int_{0}^{\pi}\left(\dfrac{\p}{\p R}F_1^+(\T,(-1+\sqrt{1-R^2})/R)r_1^+(\T,R)\right)d\T,\\
\mathcal{G}^{-}(R)&=\int_{0}^{\pi}\left(\dfrac{\p}{\p R}F_1^-(\T-\pi,(1-\sqrt{1-R^2})/R)r_1^-(\T-\pi,R)\right)d\T,
\end{aligned}
\end{equation*}
can be computed using the expressions ${\{s,c\}}^{\lambda}_{k,\ell},$ ${\{s,c\}}^{\phi}_{k,\ell},$ ${\{s,c\}}^{\T}_{k,\ell},$ ${\{s,c\}}_{k,\ell}^{\alpha=0},$ for $\ell=3$ with $k\in\{0,1,2,3,4\},$ and ${\{s,c\}}_{k,\ell}^{\alpha=0}$ for $\ell=4$ with $k\in\{0,1,2,3,4,5,6\}.$ Finally, taking $R=-2r/(1+r^2)$ and $R=2r/(1+r^2)$ in $\mathcal{G}^+(R)$ and $\mathcal{G}^-(R),$ respectively, we get back the original variable $r$. Hence, the second averaged function writes
\begin{equation}\label{eq2o}
\CF_2(r)=\CF_2^{(1)}(r)+\CF_2^{(2)}(r)=\CF_2^{(1)}(r)+\mathcal{G}^+(-2r/(1+r^2))+\mathcal{G}^-(2r/(1+r^2)).
\end{equation}
Now, from Lemma~\ref{lem 5.1}, expression \eqref{mels1}, and applying the change of parameters
\[
\begin{aligned}
c_0=& p^-_{{1,1,1}}+p^+_{{1,1,1}}+2\,q^-_{{1,2,0}}+2\,q^+_{{1,2,0}},\\
c_1=& 2\,p^+_{{1,0,1}}-2\,p^+_{{1,0,2}}-2\,p^+_{{1,2,0}}+2\,q
^-_{{1,1,0}}-q^-_{{1,1,1}}+q^+_{{1,1,1}},\\
c_2=&p^+_{{1,1,0}}+q^+_{{1,0,1}},\\
c_3=&-2\,p^+_{{1,1,0}}+2\,q^+_{{1,0,1}}-2\,q^+_{{1,0,2}}-2\, q^+_{{1,2,0}},
\end{aligned}
\]
the second averaged function \eqref{eq2o} writes
\begin{equation}\label{eq2oo}
\begin{aligned}
\CF_2(r)=&\frac{H_0(r)}{r^2}+ \frac{(1-r^2)H_1(r)}{r^3} \log\left(\frac{1-r}{1+r}\right)\\
& +H_2(r) \log\left(\frac{(1-r^2)^2(1+r^2)}{(1+6r^2+{r}^{4})^3}\right)+ H_3(r)\Phi_0^{0}\left(\frac{2r}{r^2+1}\right),
\end{aligned}
\end{equation}
where $\Phi_0^{0}$ is defined in \eqref{eq:12a}, $H_0(r)$ and $H_1(r)$ are polynomials of degrees 7 and 6, respectively, satisfying 
\begin{equation}\label{eq:hh01}
H_0(0)=2H_1(0), \quad H'_0(0)=2H'_1(0), \quad H''_0(0)=-\frac{8}{3}H_1(0)+2 H''_1(0).
\end{equation}
Moreover, $H_0$ and $H_1$ depend quadratically on the parameters $c_i, p_{1,i,j}^\pm, q_{1,i,j}^\pm$ and linearly on the new parameters $k_i.$ The other two functions write
\[
\begin{aligned}
H_2(r)&=-\frac { \left( r^2-1 \right) ^{2} \pi }{16r}\,c_0 \,c_1 ,\\
H_3(r)&= \frac { r^2-1 }{8
	r \left( {r}^{2}+1 \right) } \, c_0\, ( c_3\,{r}^{4}-2\, c_3\,{r}^{2}-2\, c_2).
\end{aligned}
\]
The above conditions imply that $\CF_2(r)=O(r).$ This concludes the first part of this proof.

\medskip
 
Now, we compute the Taylor series of $\CF_2$ given in \eqref{eq2oo} around $r=0$. So,
\begin{equation*}\label{eq:13}
\CF_2(r)=\sum_{i=0}^{n}f_{i}r^{i+1}+O(r^{n+2}).
\end{equation*}
Here, the coefficients $f_i$ are quadratic functions in the variables $c_i, p_{1,i,j}^\pm, q_{1,i,j}^\pm$ and linear in the variables $k_i, p^+_{1,0,1}, q^+_{1,0,2},$ $q^-_{1,2,0}\}.$ Assuming that $c_0=1,$ $2c_{2}+ c_{3}-2{ p}^+_{{1,1,1}}-4{\it q}^-_{{1,0,2}}-4q^+_{{1,0,2}}=1,$ $116c_{2}+58c_{3}-116{p}^+_{{1,1,1}}-230{ q}^-_{{1,2,0}}-230{q}^+_{{1,2,0}}-59=1,$ $-2c_{2}-c_{3}+2{p}^+_{{1,1,1}}+5=1,$ the system of equations
\begin{equation*}
\{f_0=c_4,\, f_1=c_5,\, f_2=c_6, \, f_3=c_7, f_4=c_9, \, f_5=c_8, \, f_7=c_{10}, \, f_9=c_{11}\}
\end{equation*}
has a unique solution. Accordingly, all the perturbation parameters depend only on $\{c_1,\ldots,c_{11}\}.$
So, the second order averaged function writes
\begin{equation*}\label{eq:12abc}
\CF_2(r)=\sum_{i=0}^{13}g_{i}r^{i+1}+O(r^{15}),
\end{equation*} 
with $g_{{0}} = c_4,$ $g_{{1}}=c_5,$ $g_{{2}} = c_6,$ $g_{{3}}=c_7,$ $g_{{4}} = c_9, $ $g_{{5}}=c_8,$ $g_{{7}} = c_{10},$ $ g_{{9}} = c_{11}$
\[
\begin{array}{rl}
g_{{6}} = &3517699860675\pi c_1 + c_4 - c_6 + c_9, \vspace{0.2cm}\\
g_{{8}} = &-12593243758095\pi c_1,\vspace{0.2cm}\\
g_{{10}} = &61075412843445\pi c_1,\vspace{0.2cm}\\
g_{{11}} =&- 786432c_2 +\dfrac{63045632}{3}c_{3} + \dfrac{3632}{1287}c_5- {\dfrac{3632}{1287}{c_7}}+ \dfrac{1211}{429}c_8 - {\dfrac{109}{39}c_{10}}+ \dfrac{7}{3}\,c_{11}, \vspace{0.2cm} \\
g_{{12}}=& -304692133550805\pi c_1,\vspace{0.2cm}\\
g_{{13}} =& - 2484794504\, c_2 + 16745167364\, c_3 + \dfrac{7088319}{36608}c_5 - \dfrac{7088319}{36608}c_7\vspace{0.2cm}\\
& 
+ \dfrac{56833457}{292864}c_8 -\dfrac{1230915}{6656}c_{10} + \dfrac{209569}{2048}c_{11}.
\end{array} 
\]
Notice that $(g_0,g_1,\ldots,g_{13})$ is a linear function on the variable $(c_1,\ldots,c_{11})$. Since its rank with respect to $(c_1,\ldots,c_{11})$ is $11$, there exists a change of variables $(c_1,\ldots,c_{11})\mapsto(d_1,\ldots,d_{11})$ such that 
\begin{equation*}
\CF_2(r)=r(\sum_{i=1}^{11}d_{i}U_{a_i}(r)+O(r^{14})),
\end{equation*} 
where $U_{a_i}(r)=r^{a_i}+O(r^{14}),$ $a_i=i-1,$ $i=1,2,\ldots,8,$ $a_9=9$, $a_{10}=11,$ and $a_{11}=13.$ Since $\CF_2$ is analytic at $r=0,$ the Weierstrass Preparation Theorem (see, for instance, \cite{ChoHal1982}) implies that there exists an analytic function $F$ such that $F(0)\ne 0$ and
\begin{equation*}
\CF_2(r)=r F(r) \sum_{i=1}^{11}d_{i}r^{a_i},
\end{equation*} 
This proof follows by noticing that the parameters $d_i$, $i=1,2,\ldots,11,$ can be chosen (small) in order that the function $\CF_2(r)$ has 10 simple zeros near the origin.
\end{proof}

\begin{rem}\label{re:dificilo2}
There are two main difficulties in studying the maximum number of simple zeros of \eqref{eq2oo}. First, $\Phi_0^{0}$ is an integral function that cannot be expressed with simple functions. Second, the parameter coefficients of the polynomials $H_0$ and $H_1$ have a quadratic dependence on the parameters of $Z_{1,\e}$, consequently the ECT-systems theory cannot be directly applied.
\end{rem}

Similar difficulties as pointed out by Remark~\ref{re:dificilo2} will also appear in the next two propositions.

\begin{prop}\label{pr:S3O2}
For $|\e|>0$ sufficiently small, the maximum number of limit cycles that $Z_{3,\e}$ can have in any neighborhood of the origin is at least 10 when the curve of discontinuity is $\{x=0\}.$
\end{prop}

\begin{proof} 
The proof follows the same steps as the proof of Proposition~\ref{pr:S1O2}. We recall that the first order analysis has been performed in Proposition~\ref{pr:S3O1}.
Again, an extra limit cycle can be obtained from a pseudo-Hopf bifurcation, so we may assume that $P^{\pm}(0,0)=Q^{\pm}(0,0)=0.$ Then, the proof will consists in applying the second order averaging method to get at least 9 limit cycles bifurcating from the origin.

Firstly, using \eqref{r1} for $\al=0$, the functions $r^{+}_1(\T,R)$ and $r^{-}_1(\T-\pi,R)$ write
\begin{equation*}
\begin{aligned}
	r^{+}_1(\T,R)&=\int_{0}^{\T}F_1^+(\T,-R/8)d\T,\\
	r^{-}_1(\T-\pi,R)&=\int_{0}^{\T}F^{-}_1(\T-\pi,R/8)d\T.
\end{aligned}
\end{equation*}	
The above integrals can be computed using the expressions from Section~\ref{se:5}, $\mathcal{\{S,C\}}^{\alpha=0}_{k,\ell}$ for $\ell=4,$ and $k=0,1,\ldots,6$. Then, the second summand of the second averaged function, $\CF_2^{(2)},$ follows from the integrals
\begin{equation*}
\begin{aligned}
\mathcal{G}^+(R)&=\int_{0}^{\pi}\left(\dfrac{\p}{\p R}F_1^+(\T,-R/8)r_1^+(\T,R)\right)d\T,\\
\mathcal{G}^{-}(R)&=\int_{0}^{\pi}\left(\dfrac{\p}{\p R}F_1^-(\T-\pi,R/8)r_1^-(\T-\pi,R)\right)d\T,
\end{aligned}
\end{equation*}
which can be computed using the expressions from Section~\ref{se:5}, ${\{s,c\}}^{\lambda}_{k,\ell},$ ${\{s,c\}}^{\phi}_{k,\ell},$ ${\{s,c\}}^{\T}_{k,\ell},$ ${\{s,c\}}_{k,\ell}^{\alpha=0}$ for $\ell=5$ with $k=0,1,\ldots,7,$ and ${\{s,c\}}_{k,\ell}^{\alpha=0}$ with $\ell=8$ and $k=0,1,\ldots,12.$ So, going back to the original variable $r$ we get
\begin{equation*}
\CF_2(r)=\CF_2^{(1)}(r)+\CF_2^{(2)}(r)=\CF_2^{(1)}(r)+\mathcal{G}^+(-8r)+\mathcal{G}^-(8r).
\end{equation*}
Again, from Lemma~\ref{lem 5.1} and expression \eqref{mels3} we get
\begin{equation*}
\begin{aligned}
\CF_2(r)= &\frac{H_0(r)}{r^2(1-r^4)^2}+ \frac{H_1(r)}{r^3(1+r^2)} \log\left(\frac{1-r}{1+r}\right).
\end{aligned}
\end{equation*}
Here, $H_0$ and $H_1$ are polynomials of degree 13 and 8, respectively, and satisfy the relations \eqref{eq:hh01}. As previously, with these conditions, we have $\CF_2(r)=O(r).$

\medskip

Now, computing the Taylor series of $\CF_2$ around $r=0$ we get
\begin{equation}\label{eq:15}
\CF_2(r)=\sum_{i=1}^{n}f_{i}r^{i}+O(r^{n+1}).
\end{equation}
In order to simplify the expression of $\CF_2(r)$, we introduce the new parameters,
\[
\begin{array}{rl}
c_1=& 16\,q^-_{1,1,0}+3\,q^-_{1,1,1}-16\,q^+_{1,1,0}-3\,q^+_{1,1,1},\vspace{0.2cm}\\
c_2=&-16\,q^-_{1,1,0}+4\,p^+_{1,0,2}-3\,q^-_{1,1,1}+16\,q^+_{1,1,0}+3\,q^+_{1,1,1},\vspace{0.2cm}\\
c_3=&-48\,q^-_{1,1,0}+12\,p^+_{1,0,2}+8\,p^+_{1,2,0}-9\,q^-_{1,1,1}+48\,q^+_{1,1,0}+7\,q^+_{1,1,1},\vspace{0.2cm}\\
c_4=& 184\,p^+_{1,1,1}+21\,q^-_{1,0,2}+71\,q^+_{1,0,2},\vspace{0.2cm}\\
c_5=&q^-_{1,0,2}+q^+_{1,0,2},\vspace{0.2cm}\\
c_6=&q^-_{1,2,0}+q^+_{1,2,0},\vspace{0.2cm}\\
c_7=& 200\,p^+_{1,1,0}+579\,p^+_{1,1,1}-825\,q^-_{1,2,0}-1300\,q^+_{1,0,1},\vspace{0.2cm}\\
c_8=& 16\,p^-_{1,0,1}-16\,p^+_{1,0,1}-3\,q^-_{1,1,1}+3\,q^+_{1,1,1}, \vspace{0.2cm}\\
c_9=& 16\,p^+_{1,0,1}-3\,q^+_{1,1,1},\vspace{0.2cm}\\
c_{10}=&q^+_{1,1,0}+p^+_{1,0,1},\vspace{0.2cm}\\
c_{11}=& 800\,p^+_{1,1,0}-1413\,p^+_{1,1,1}-800\,q^+_{1,0,1},\vspace{0.2cm}\\
c_{12}=& -\dfrac{789}{6125}p^+_{1,1,1}-\dfrac{1}{24500}
c_4+\dfrac{71}{24500}c_5+\dfrac{184}{18375}c_6+\dfrac{4}{67375}c_7-\dfrac{2}{28875}c_{11}.
\end{array}
\]
We see that the coefficients $f_i$ in \eqref{eq:15} depend linearly on the second order parameters, $p^\pm_{2,i,j}$ and $q^\pm_{2,i,j},$ and quadratically on the new parameters $c_i.$ Under the assumption $P^{\pm}(0,0)=Q^{\pm}(0,0)=0,$ the first averaged function, studied in Proposition~\ref{pr:S3O1}, provides only the first 5 linearly independent coefficients. Thus, there exists a transformation on the parameters space such that $f_i=d_i$ for $i=1,\ldots,5,$ and, under the condition $c_1=1,$ $c_5=c_6=c_8=c_{10}=0,$ the system 
\[
\{f_{6}=d_{6},f_{7}=d_{7},f_{8}=d_{8},f_{9}=d_{9},f_{10}=d_{10}\}
\]
has a unique solution with respect to $\{c_2,c_3,c_4,c_7,c_9\}.$
Therefore, we get 
\begin{equation*}\label{eq:f2o2s2}
\CF_2(r)=\sum_{i=1}^{10} d_ir^{i}+O(r^{11}).
\end{equation*}
In this case, the multiplicity of the origin cannot be increased because the coefficients $d_{11}$ and $d_{12}$ depend linearly on the parameters $\{d_1,\ldots,d_{10}\}.$ In fact, 
\[
\begin{aligned}
d_{11}&=d_5+d_7-d_9\\
d_{12}&=-\frac {76}{715}d_2-\frac {37}{65}d_4+\frac {502}{715}d_6+\frac {116}{65}d_8-\frac{2}{5}d_{10}.
\end{aligned}
\]
Finally, this proof follows by noticing that the parameters $d_i$, $i=1,2,\ldots,9,$ can be chosen (small) in order that the function $\CF_2(r)$ has 9 simple zeros near the origin.
\end{proof}

Next technical results, whose proofs are straightforward, provide lower and upper bounds for the values that a polynomial of $n$ variables take in a $n$ dimensional polyhedron. Moreover, they will be useful for proving the last proposition of this section concerning the second order analysis of the system $Z_{2,\e}$.

\begin{lem}\label{le:sigmachi}
Consider $h>0,$ $p>0,$ $q$ real numbers such that $p\in[\underline{p},\overline{p}],$ with $\underline{p}\overline{p}>0,$ and $q\in[\underline{q},\overline{q}],$ with $\underline{q}\overline{q}>0.$ 
\begin{enumerate}[(i)]
\item Then, $\sigma^\ell(q,p)\leq qp\leq \sigma^r(q,p),$ \\[5pt]
where $\sigma^\ell(q,p)=\begin{cases} q \, \underline{p}, & \text{if } q>0, \\ q \, \overline{p} , & \text{if } q<0, \end{cases}$ and $\sigma^r(q,p)=\begin{cases} q \, \overline{p}, & \text{if } q>0, \\ q \, \underline{p}, & \text{if } q<0. \end{cases}$\\[2pt]
\item If $u_j\in[-h,h],$ for $j=1,\ldots,n,$ and denoting $u^i=u_1^{i_1}u_2^{i_2}\cdots u_n^{i_n}$ for $i=(i_1,\ldots,i_n)\ne 0,$ we have $\chi^\ell(q,u^i)\leq q u^i\leq \chi^r(q,u^i),$ where \\[5pt]
$\chi^\ell(q,u^i)=\begin{cases} 0, & \text{if } q>0 \text{ and }i_k\text{ even for all } k=1,\ldots,n, \\ -\overline{q} \, h^{i_1+\cdots+i_n} , & \text{if } q>0 \text{ and } i_k \text{ odd for some } k=1,\ldots, n,\\ \underline{q} \, h^{i_1+\cdots +i_n} , & \text{if } q<0, \end{cases}$ \newline 
and \newline
$\chi^r(q,u^i)=\begin{cases} \overline{q} \, h^{i_1+\cdots+i_n}, & \text{if } q>0, \\ 0, & \text{if } q<0 \text{ and }i_k\text{ even for all } k=1,\ldots,n,\\ -\underline{q} \, h^{i_1+\cdots +i_n} , & \text{if } q<0 \text{ and } i_k \text{ odd for some } k=1,\ldots, n. \end{cases}$\newline
Furthermore, $\chi^\ell(q,1)=\underline{q}$ and $\chi^r(q,1)=\overline{q}.$
\end{enumerate}
\end{lem}

\begin{lem}\label{le:algoritmo}
Let $h>0$ and $p_j$ be a positive non rational numbers such that $p_j\in [\underline{p_j},\overline{p_j}]$ with $\underline{p_j},\overline{p_j}$ rational numbers satisfying $\underline{p_j}\overline{p_j}>0,$ for $j=1,\ldots,m.$ Consider the polynomial 
\begin{equation}\label{eq:Ualg}
\mathcal U(u_1,\ldots,u_n)=\sum_{i_1+\cdots+i_n=0}^{M}
\biggl(\sum_{j=1}^{m}U_{j,i}\, p_j\biggr) u^i,
\end{equation}
with $u^i=u_1^{i_1}\cdots u_n^{i_n},$ $i=(i_1,\ldots,i_n),$ and $U_{j, i}$ rational numbers. Then,
\[
U^\ell_{i}\leq \sum_{j=1}^{m}U_{j,i}\, p_j\leq U^r_{i}
\]
with $U^\ell_{i}=\sum_{j=1}^{m}U_{j,i} \cdot \sigma^\ell (U_{j,i},p_j)$ and $U^r_{i}=\sum_{j=1}^{m}U_{j,i}\cdot \sigma^r (U_{j,i},p_j).$ Moreover, if $u_j\in[-h,h],$ for $j=1,\ldots,n,$ and $U^\ell_{i}\cdot U^r_{i}>0$ then 
\[
\underline{\mathcal U}=\sum_{i_1+\cdots+i_n=0}^{M}
\chi^\ell(U^\ell_{i},u^i)\leq \mathcal U(u_1,\ldots,u_n)\leq \sum_{i_1+\cdots+i_n=0}^{M}
\chi^r(U^r_{i},u^i)=\overline{\mathcal U}.
\]
\end{lem}

The next example shows how the above two technical lemmas can be used to get rational lower and upper bounds for the values that a function takes in a given 3D-polyhedron. In this example the lower and upper bounds for the values of $\pi$ and $\sqrt{3}$ are chosen from their continued fraction.

\begin{example}\label{ex:algoritmo} Consider the polynomial 
\begin{equation}\label{poli}
P(u,v,w)=P_0+P_1 \,u+P_2 \,v^2w^2 + P_3\, uv^2w^4,
\end{equation}
with $P_{0}=\pi-5\sqrt{3}+4,$ $P_{1}=-\pi^2+3\sqrt{3}-3,$ $P_{2}=-2\pi^3-\sqrt{3}+70,$ and $P_{3}=4\pi+\sqrt{3}+7.$ 
Consider the following intervals containing $\pi$ and $\sqrt{3},$
\begin{equation*}
\pi\in\left[\underline{p},\overline{p}\right]=\left[\frac{333}{106},\frac{355}{113}\right] \quad \text{and} \quad \sqrt{3}\in\left[\underline{s},\overline{s}\right]=\left[\frac{5}{3},\frac{7}{4}\right].
\end{equation*} 
Then, for $u,v,w\in[-1/9,1/9],$ we have
\[
P(u,v,w)\in\bigg[-\frac{8036904331130}{3236907751533},-\frac{5753192708807927}{18184947748112394}\bigg].
\]
\end{example}

\begin{proof}Following the notation of Lemma~\ref{le:algoritmo}, we take $p_1=\pi,$	$p_2=\pi^2,$ $p_3=\pi^3,$ and $p_4=\sqrt{3}.$ Then, the coefficients of \eqref{poli} write as $P_0=p_1-5p_4+4,$ $P_1=-p_2+3p_4-3,$ $P_2=-2p_3-p_4+70,$ and $P_3=4p_1+p_4+7.$ 

The intervals given in the statement provide that $\{p_1,p_4,p_2\}\subset[\underline{p}^2,\overline{p}^2]$ and $p_3\in[\underline{p}^3,\overline{p}^3].$ So, these new variables allow us to define the coefficients $U_{j,i}$ in \eqref{eq:Ualg}. From Lemma~\ref{le:sigmachi}(i) we have
\begin{equation*}
\begin{aligned}
-\frac{341}{212}=\underline{p} -5\overline{s}+4=U_{000}^\ell\leq &P_0 \leq U_{000}^r=\overline{p} -5\underline{s}+4=-\frac{404}{339},\\
-\frac{100487}{12769}=-\overline{p}^2+3\underline{s}-3=U_{100}^\ell\leq &P_1 \leq U_{100}^r=-\underline{p}^2 +3\overline{s}-3= -\frac{21402}{2809},\\
\frac{35999881}{5771588}=-2\overline{p}^3 -\overline{s}+70=U_{022}^\ell\leq &P_2 \leq U_{022}^r= -2\underline{p}^3-\underline{s}+70 = \frac{11301029}{1786524},\\
\frac{3376}{159}=4\underline{p}+\underline{s}+7=U_{124}^\ell\leq &P_3 \leq U_{124}^r=4\overline{p}+\overline{s}+7 = \frac{9635}{452}.\\
\end{aligned}
\end{equation*}
Finally, applying Lemma~\ref{le:sigmachi}(ii) we get the following lower and upper bounds for $P,$
\begin{equation*}
U_{000}^\ell +U_{100}^\ell h - U_{124}^r h^7=\underline{\mathcal U}\leq P (u,v,w) \leq \overline{\mathcal U}=U_{000}^r -U_{100}^\ell h + U_{022}^r h^4 + U_{124}^r h^7.
\end{equation*}
The proof follows by substituting $h=1/2$ in the above expression.
\end{proof}

The last proposition deals with second order perturbation of family $S_2,$ which exhibits the highest number of limit cycles found in this paper. Theorem~\ref{thm:Main} is a direct consequence of it. In Proposition~\ref{pr:S2O1} we have studied the zeros of the first averaged function for three different straight lines of discontinuity. The best result was obtained when $\Sigma=\{y+\sqrt{3}x=0\}$. So, we shall perform the second order analysis only in this case.

\begin{prop}\label{pr:S2O2}
For $|\e|>0$ sufficiently small, the maximum number of limit cycles that $Z_{2,\e}$ can have in any neighborhood of the origin is at least 16 when the curve of discontinuity is $\{y+\sqrt{3}x=0\}.$
\end{prop}

\begin{proof}
The proof follows basically the same steps as the proofs of Propositions~\ref{pr:S1O2} and \ref{pr:S3O2}. Nevertheless, in this case, some of the integrals of $\CF_2$ cannot be explicitly obtained. Then, since the functions are analytic near the origin, we compute the Taylor series of the integrand before integrating. We recall that the first order analysis has been performed in Proposition~\ref{pr:S2O1}. Again, an extra limit cycle can be obtained from a pseudo-Hopf bifurcation, so we may assume that $P^{\pm}(0,0)=Q^{\pm}(0,0)=0.$ Then, the proof will consists in applying the second order averaging method to get at least 15 limit cycles bifurcating from the origin.

Using \eqref{r1} for $\alpha=-\pi/3,$ the functions $r^{+}_1(\T,R)$ and $r^{-}_1(\T-\pi,R)$ write \begin{equation*}
r^{+}_1(\T,R)=\int_{-\pi/3}^{-\pi/3+\T}F_1^+(\T,R)d\T,
\end{equation*}
and
\begin{equation*}
r^{-}_1(\T-\pi,R)=\int_{-\pi/3}^{-\pi/3+\T}F^{-}_1(\T-\pi,-R)d\T.
\end{equation*} 
The above integrals can be computed using the expressions for $\mathcal{\{S,C\}}_{k,\ell}^{\alpha=-\pi/3},$ for $\ell=1$ and $k=0,1,2,3,$ from Section~\ref{se:5}. The second summand of the second averaged function, $\CF_2^{(2)},$ follows from the integrals
\begin{equation*}
\begin{aligned}
\mathcal{G}^+(R)&=\int_{-\pi/3}^{2\pi/3}\bigg(\dfrac{\p}{\p R}F_1^+(\T,R)r_1^+(\T,R) \bigg)d\T,\\
\mathcal{G}^-(R)&=\int_{-\pi/3}^{2\pi/3} \bigg(\dfrac{\p}{\p R}F_1^-(\T-\pi,-R)r_1^-(\T-\pi,-R)\bigg)d\T.
\end{aligned}
\end{equation*}
We point out that the integrands of the above integrals are rational functions with denominators $(1+R\cos\T)^2$ and numerators depending on 
\[
\{R,\theta, \cos\T, \sin\T, \lambda(R, -\pi/3 + \T), \lambda(R, -4\pi/3 + \T), \phi(R, -\pi/3 + \T), \phi(R, -4\pi/3+ \T)\}.
\]
Computing the Taylor series around $R=0,$ integrating on the interval $[-\pi/3,2\pi/3]$, and going back to the original variable $r, $ we get the Taylor series around $r=0$ of $\CF_2^{(2)}(r)$. The Taylor series of $\CF_2^{(1)}(r)$ is obtained analogously to the series of $\CF_1(r)$ in Proposition~\ref{pr:S2O1}. Accordingly, the second averaged function writes
\begin{equation*}
\CF_2(r)=\sum_{i=1}^{n} f_ir^i+O(r^{n+1}).
\end{equation*}
The coefficients $f_i$'s depend linearly on $\{p^\pm_{2,i,j},q^\pm_{2,i,j}\}$ and quadratically on $\{p^\pm_{1,i,j},q^\pm_{1,i,j}\}.$ 

\medskip

The rest of the proof is devoted to show that there exists a transformation on the parameters such that the above function becomes
\[
\CF_2(r)=\sum_{i=1}^{16} d_ir^i+O(r^{n+1}),
\]
where $d_1,\ldots,d_{16}$ are independent parameters. In fact, we shall prove the existence of a transversal curve of weak foci of order 16. The transversality also guarantees the unfolding of 15 simple zeros near the origin because our function, also the perturbed one, vanishes at zero. The existence of such curve is obtained in two steps. Firstly, we analyse the maximal rank $(f_1,\ldots,f_n)$ with respect to the linear parameters $\{p^\pm_{2,i,j},q^\pm_{2,i,j}\}.$ Secondly, proceeding with a change of parameters, which eliminates the linear terms, we study the quadratic terms regarding $\{p^\pm_{1,i,j},q^\pm_{1,i,j}\}$ from $(f_1,\ldots,f_n)$. We shall see that these quadratic terms are homogeneous and we show the existence of a transversal straight line such that these terms vanish on it. The described procedure is detailed in \cite[Theorems. 2.1 and 3.1]{Chr2006}. These ideas have been originally introduced in \cite{ChiJac1989,ChiJac1991} for quadratic vector fields and have also been employed in \cite{Han1999} for Li\'enard families.

Firstly, we see that the system of equations 
\[
\{f_1=d_1,\ldots,f_6=d_6,f_8=d_8,f_{10}=d_{10}\}
\]
has a unique solution with respect to the variables $p^{+}_{2,1,0}, p^{+}_{2,0,2},$ $p^{+}_{2,1,1}, p^{+}_{2,2,0}, q^{+}_{2,2,0}, q^{+}_{2,0,1},$ $p^{-}_{2,1,1},$ and $q^{-}_{2,2,0}.$ Hence, after this change, the rest of coefficients remains quadratic.

Secondly, assuming $f_{16}\ne0$ we shall obtain a transversal solution of the quadratic system 
\[
\mathcal{S}: \{f_7=f_9=f_{11}=f_{12}=f_{13}=f_{14}=f_{15}=0\}.
\]
Since there are more parameters than necessary, we impose that $\{p^-_{1,2,0}=p^+_{1,1,1}=q^-_{1,0,1}=q^-_{1,0,2}=q^-_{1,1,0}=q^-_{1,1,1}=q^-_{1,2,0}=q^+_{1,0,1}=q^+_{1,0, 2}=q^+_{1,1,0}=q^+_{1,1,1}=q^+_{1,2,0}=0,p^+_{1,2,0}=1\}.$ Furthermore, it is not restrictive to assume that the first parameters $d_1,d_2,\ldots,d_6,d_8,$ and $d_{10}$ vanish. For the sake of simplicity, we change the names of the remaining parameters $[p^-_{1,0,1},p^-_{1,1,0},p^+_{1,0,1},p^+_{1,1,0},p^-_{1,0,2},p^-_{1,1,1},$ $p^+_{1,0,2}]$ to $[z_1,\ldots,z_7].$ In order to solve the quadratic system $\mathcal{S}$ we consider two quadratic subsystems, namely $\mathcal{S}_1=\{f_7=f_9=f_{11}=f_{13}=f_{15}=0\}$ and $\mathcal{S}_2=\{f_{12}=f_{14}=f_{15}=0\}.$ Then, we study the intersection between their solutions. 

Using the condition $f_{15}=0,$ the subsystem $\mathcal{S}_1$ can be rewritten in order that all the equations depend linearly on $z_1,z_2,z_3,$ and $z_4.$ So, solving $\mathcal{S}_1$ in these parameters we get
\[
z_i=\dfrac{\zeta_i(z_5,z_6,z_7)}{\eta(z_5,z_6,z_7)}, \quad \text{for} \quad i=1,2,3,4,
\]
where $\zeta_i$ are polynomials of degree 5 and $\eta$ is a polynomial of degree 4. Later on, we shall see that $\eta$ does not vanish at the intersection point.
Accordingly, the parameters $f_{12},f_{14},$ and $f_{15}$ write
\begin{equation*}\label{eq:ft12ft14ft15}
f_{12}=\frac{\widetilde{f}_{12}(z_5,z_6,z_7)}{(\eta(z_5,z_6,z_7))^2},
\quad 
f_{14}=\frac{\widetilde{f}_{14}(z_5,z_6,z_7)}{(\eta(z_5,z_6,z_7))^2},
\quad
f_{15}=\frac{\widetilde{f}_{15}(z_5,z_6,z_7)}{(\eta(z_5,z_6,z_7))^2}.
\end{equation*}
where $\widetilde{f}_{12}, \widetilde{f}_{14},$ and $\widetilde{f}_{15}$ are polynomials of degree 10. So, on the variety provided by $\mathcal{S}_1$, the subsystem $\mathcal{S}_2$ is equivalent to
\begin{equation*}\label{newsystem}
\widetilde{\mathcal{S}}_2:\{\widetilde{f}_{12}(z_5,z_6,z_7)=0,\widetilde{f}_{13}(z_5,z_6,z_7)=0,\widetilde{f}_{14}(z_5,z_6,z_7)=0\},
\end{equation*}
provided that $\eta(z_5,z_6,z_7)\neq0.$ Consequently, the system $\mathcal{S}$ is reduced to $\widetilde{\mathcal{S}}_2$ whenever $\eta(z_5,z_6,z_7)\neq0.$ Although $\widetilde{\mathcal{S}}_2$ has only 3 equations and 3 unknowns, the high degree of these equations is a barrier for solving the system. Furthermore, the algebraic varieties provided by each equation of $\widetilde{\mathcal{S}}_2$ are numerically close to each other (see Figure~\ref{fi:varorigcoords}), which adds an extra numerical difficult.

\begin{figure}[h]
	\begin{center}
		\begin{overpic}[height=5cm]{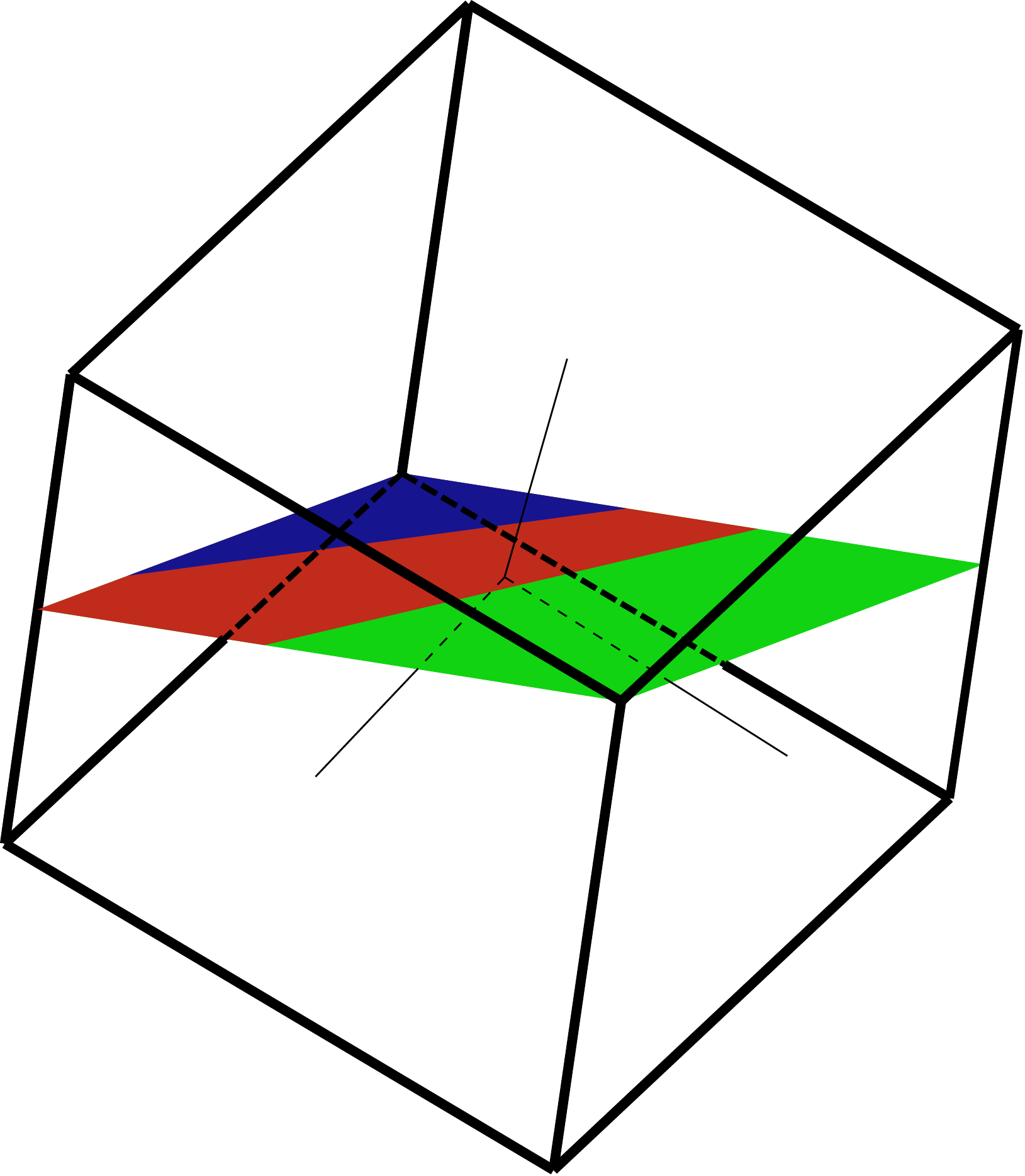}
			\put(46,75){$z_5$}
			\put(20,26){$z_6$}
			\put(62,29){$z_7$}
		\end{overpic}
	\end{center}
\caption{Plot of the varieties $\widetilde{f}_{12}=0,$ $\widetilde{f}_{14}=0,$ and $\widetilde{f}_{15}=0$ in a cube centered at $(z_5^*,z_6^*,z_7^*)$ with edges of length $10^{-2}.$ They are depicted in red, blue, and green, respectively.}\label{fi:varorigcoords}
\end{figure}

In order to overcome these difficulties, we shall first work with numerical approximations of the solutions. Then, using Lemmas~\ref{le:sigmachi} and \ref{le:algoritmo}, and Theorem~\ref{thm:PoincareMiranda} we prove analytically the numerical results. Working with enough precision we get the following numerical solution of the system $\widetilde{\mathcal{S}}_2$,
\[
(z_5^*,z_6^*,z_7^*)\approx(-0.260976000571,0.111582119099,-0.84487667629841).
\]
Moreover, we see that $\eta(z_5^*,z_6^*,z_7^*)\approx 1.85317498452382 \cdot 10^{-11}$ and $\widetilde{f}_{16}(z_5^*,z_6^*,z_7^*)=8.8767062451915\cdot 10^{-26},$ so $f_{16}=\widetilde{f}_{16}(z_5,z_6,z_7)/(\eta(z_5,z_6,z_7))^2\neq0.$ Additionally, the intersection is transversal because the determinant of the Jacobian matrix of $f=(\widetilde{f}_{12},\widetilde{f}_{14},\widetilde{f}_{15})$ with respect to $(z_5,z_6,z_7)$ evaluated at the solution $(z_5^*,z_6^*,z_7^*)$ does not vanish. In fact, $J_f(z_5^*,z_6^*,z_7^*)\approx -5.379835263496\cdot 10^{-67}.$ It is worthwhile to mention that although the values for $\eta,$ $\widetilde{f}_{16},$ and $J_f$ are very small at $(z_5^*,z_6^*,z_7^*),$ we were able to observe that they remain fixed when we increase the precision of the computations, while the values for $\widetilde{f}_{12}, \widetilde{f}_{14},$ and $\widetilde{f}_{15}$ decrease to zero. 

\medskip

Finally, we shall prove analytically the existence of such transversal intersection point $(z_5^*,z_6^*,z_7^*).$ In order to apply Lemmas~\ref{le:sigmachi} and \ref{le:algoritmo}, we make the following change of variables, 
\begin{equation*}
\begin{aligned}
z_5=&-\frac{102563793961}{75692301} \, u_1+\frac{93673471843}{117235838} \, u_2-\frac{5228323783}{13687494949} \, u_3-\frac{1104348344}{4231608813},\\[5pt]
z_6=&-\frac{114951879798}{118751113} \, u_1+\frac{66846520379}{116131808} \, u_2+\frac{1728113446}{4218432187} \, u_3+\frac{859801297}{7705547304},\\[5pt]
z_7=&\frac{131538341646}{188147809} \, u_1-\frac{23870722389}{57947275} \, u_2+\frac{1387092713}{5464980203} \, u_3-\frac{2790022856}{3302284149}.
\end{aligned}
\end{equation*}
Then, 
\begin{equation*}
\begin{aligned}
\widetilde{f}_{12}&=\xi_1(u_1,u_2,u_3),&
\widetilde{f}_{14}&=\xi_2(u_1,u_2,u_3),&
\widetilde{f}_{15}&=\xi_3(u_1,u_2,u_3),\\
\widetilde{f}_{16}&=\xi_4(u_1,u_2,u_3),&
\eta&=\xi_5(u_1,u_2,u_3),&
J_f&=\xi_6(u_1,u_2,u_3).
\end{aligned}
\end{equation*}
Now, the problem is to check that the varieties defined by $\xi_1=0,$ $\xi_2=0,$ and $\xi_3=0$ intersect transversally near the origin (see Figure~\ref{fi:varnewcoords}).
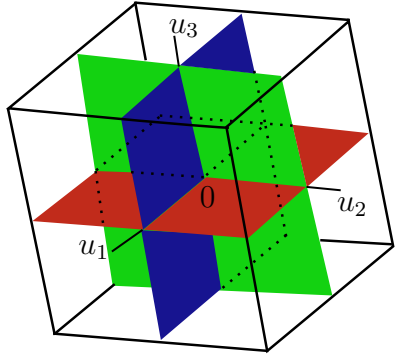
\begin{figure}[h]
	\begin{center}
		\begin{overpic}{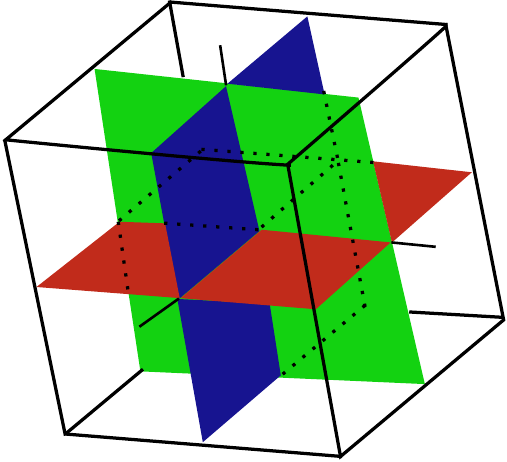}
			\put(42,83){$u_3$}
			\put(19,26){$u_1$}
			\put(85,37){$u_2$}
			\put(50,38){$0$}
		\end{overpic}
	\end{center}
	\caption{Plot of $\xi_1=\xi_2=\xi_3=0$ in a neighborhood of the origin. The varieties are drawn in green, blue, and red, respectively. The length of the edges of the 3d-cube is $10^{-10}.$}\label{fi:varnewcoords}
\end{figure} 

We notice that the functions $\xi_i,$ $i=1,\ldots,6,$ are polynomials in $(u_1,u_2,u_3).$ Moreover, $\xi_i$, $i=1,\ldots,4$ have degree 10, $\xi_5$ has degree 4, and $\xi_6$ has degree 27. We see that the coefficients of the previous polynomials depend on the irrational numbers $\pi$ and $\sqrt{3}$. More specifically, the coefficients of $\xi_i,$ $i=1,\ldots,4,$ depend on $\pi$ up to power 12, $\xi_5$ depends on $\pi$ up to power 5, and $\xi_6$ depends on $\pi$ up to power 37. The number $\sqrt{3}$ appears in these coefficients with no exponent. For each function $\xi_i,$ let $\mu_i$ denote the degree with respect to $\pi.$ Then, substituting $\sqrt{3},\pi,\pi^2,\ldots,\pi^\mu_i$ by $p_1,p_2,\ldots,p_{\mu_i+1},$ respectively, all the coefficients of the polynomial $\xi_i(p_1,p_2,\ldots,p_{\mu_i+1},u_1,u_2,u_3)$ are now rational numbers defined as the quotient of two big integers, around 1000 figures each. Moreover, they have $6292,$ $6292,$ $6006,$ $6292,$ $220$ and $234668$ monomials. 
 
Finally, we can apply Theorem~\ref{thm:PoincareMiranda} to $(\xi_1,\xi_2,\xi_3)$ in $B=[-h,h]^3,$ with $h=10^{-10},$ because
\begin{equation*}
\begin{aligned}
\xi_1(B_1^-)&\subset[a_1^-,b_1^-]\approx [-3.58842524 \cdot 10^{-32},-3.20226086 \cdot 10^{-32}],\\
\xi_1(B_1^+)&\subset[a_1^+,b_1^+]\approx [ 3.28009013 \cdot 10^{-32}, 3.66625367 \cdot 10^{-32}],\\
\xi_2(B_2^-)&\subset[a_2^-,b_2^-]\approx [-3.70185379 \cdot 10^{-32},-3.04215854 \cdot 10^{-32}],\\
\xi_2(B_2^+)&\subset[a_2^+,b_2^+]\approx [ 3.16666084 \cdot 10^{-32}, 3.82635712 \cdot 10^{-32}],\\
\xi_3(B_3^-)&\subset[a_3^-,b_3^-]\approx [-4.72369496 \cdot 10^{-32}, -1.41476503 \cdot 10^{-32}],\\
\xi_3(B_3^+)&\subset[a_3^+,b_3^+]\approx [ 2.14481860 \cdot 10^{-32}, 5.45375151 \cdot 10^{-32}].\\
\end{aligned}
\end{equation*}
We notice that the values for $a_i^\pm$ and $b_i^\pm,$ for $i=1,2,3,$ are all rational numbers explicitly computed using Lemmas~\ref{le:sigmachi}, \ref{le:algoritmo}, and
\[
\pi\in\bigg[\frac{21053343141}{6701487259},\frac{1783366216531}{567663097408}\bigg], \quad \sqrt{3}\in\bigg[\frac{716035}{413403},\frac{978122}{564719}\bigg].
\]
Hence, we conclude that there exists a point $(\xi_1^*,\xi_2^*,\xi_3^*)$ in $B$ such that $\xi_1=\xi_2=\xi_3=0.$ Additionally, $\xi_4,$ $\xi_5,$ $\xi_6$ do not vanish on $B.$ Indeed,
\begin{equation*}
\begin{aligned}
\xi_4(B)&\subset[a_4^-,b_4^-]\approx [ 8.87669600 \cdot 10^{-26}, 8.87671664 \cdot 10^{-26}],\\
\xi_5(B)&\subset[a_5^-,b_5^-]\approx [ 1.85317477 \cdot 10^{-11}, 1.85317520 \cdot 10^{-11}],\\
\xi_6(B)&\subset[a_6^+,b_6^+]\approx [-5.37983643 \cdot 10^{-67},-5.37983443 \cdot 10^{-67} ],\\
\end{aligned}
\end{equation*}
Again, the values for $a_i^\pm$ and $b_i^\pm,$ $i=4,5,6,$ are rational numbers explicitly computed using Lemmas~\ref{le:sigmachi} and \ref{le:algoritmo}. So, we conclude that the solution provided by Theorem~\ref{thm:PoincareMiranda} is a transversal solution of $\mathcal{S}$. 
\end{proof}

\section{Appendix: Explicit computations of the integrals}\label{se:5}
This section is devoted to provide explicit expressions for some of the integrals necessary to compute the averaged functions. We also introduce new special integral functions as well as some of their properties and relations. The proofs follow closely the results from \cite{ProTor2014}.

\smallskip

For each pair of natural numbers $k$ and $\ell,$ we define the following functions:
\begin{equation}\label{eq:11}
\begin{aligned}
\mathcal{S}^\alpha_{k,\ell}(r,\theta)&=\int_{\alpha}^{\alpha+\theta}
\dfrac{\sin(k\psi)}{(1+r\cos\psi)^{\ell}}d\psi, &
\mathcal{C}^\alpha_{k,\ell}(r,\theta)&=\int_{\alpha}^{\alpha+\theta}
\dfrac{\cos(k\psi)}{(1+r\cos\psi)^{\ell}}d\psi,
\\
{s}^\alpha_{k,\ell}(r)&=\int_{\alpha}^{\alpha+\pi}
\dfrac{\sin(k\T)}{(1+r\cos\T)^{\ell}}d\T,
&\!\!\!\!
{c}^\alpha_{k,\ell}(r)&=\int_{\alpha}^{\alpha+\pi}
\dfrac{\cos(k\T)}{(1+r\cos\T)^{\ell}}d\T,
\\
{s}^{\lambda}_{k,\ell}(r)&=\int_{0}^{\pi}
\dfrac{\sin(k\T)\,\lambda(r,\T)}{(1+r\cos\T)^{\ell}}d\T,
&\!\!\!\!
{c}^{\lambda}_{k,\ell}(r)&=\int_{0}^{\pi}
\dfrac{\cos(k\T)\,\lambda(r,\T)}{(1+r\cos\T)^{\ell}}d\T.
\\
{s}^{\phi}_{k,\ell}(r)&=\int_{0}^{\pi}
\dfrac{\sin(k\T)\,\phi(r,\T)}{(1+r\cos\T)^{\ell}}d\T,
&\!\!\!\!
{c}^{\phi}_{k,\ell}(r)&=\int_{0}^{\pi}
\dfrac{\cos(k\T)\,\phi(r,\T)}{(1+r\cos\T)^{\ell}}d\T,
\\
{s}^{\T}_{k,\ell}(r)&=\int_{0}^{\pi}
\dfrac{\T\sin(k\T)}{(1+r\cos\T)^{\ell}}d\T,
&\!\!\!\!
{c}^{\T}_{k,\ell}(r)&=\int_{0}^{\pi}
\dfrac{\T\cos(k\T)}{(1+r\cos\T)^{\ell}}d\T.
\end{aligned}
\end{equation}
Here, $r\in(-1,1),$ $\theta\in[-\pi,\pi],$ and $\phi,$ $\lambda$ are the two periodic functions
\begin{equation}\label{eq:12b}
\begin{aligned}
\phi(r,\theta) &=\dfrac{1}{\sqrt{1-r^2}}\left(\theta -2 \arctan\left(\sqrt{\dfrac{1-r}{1+r}}\tan\left(\dfrac{\theta}{2}\right) \right)\right),\\
\lambda(r,\theta)&=\log(1+r\cos\theta).\\
\end{aligned}
\end{equation}
\begin{lem}\label{lem 5.1} Let $\phi,\lambda$ be the functions defined by \eqref{eq:12b}. Then, 
$\phi(r,0)=\phi(-r,0)=0,$ $\phi(r,\pi)=\phi(r,-\pi)=0,$ and $\phi(-r,t+\pi)=\phi(r,t).$
Moreover, 
\[
\begin{aligned}
\frac{\partial}{\partial r}\phi(r, \T) = &\frac{r\phi(r, \T)}{1-r^2}+\frac{\sin\T}{(1-r^2)(1+r\cos\T)},\\
\frac{\partial}{\partial r}\lambda(r,\T)=&\frac{\cos\T}{1+r\cos \T}.
\end{aligned}
\]
\end{lem}
\begin{proof}
The first properties follow simply by substituting $\theta=0,\pi$ in the definition of $\phi.$ The last is satisfied because $F(r,0)=0$ and the derivative, with respect to $\theta,$ of $F(r,\theta)=\phi(-r,\theta+\pi)-\phi(r,\theta)$ vanishes identically. The expressions of their derivatives are easily to be checked.
\end{proof}
We notice that not all the above integrals can be explicitly obtained. So, next lemma introduce some new functions. They, together with their derivatives, are useful for the proofs of the results.

\begin{lem}\label{lem 5.1a} Let $\Phi_0^{0}(r),$ $\Lambda_0^{0}(r),$ and $\Lambda_1^{0}(r)$ be the functions
\begin{equation}\label{eq:12a}
\begin{aligned}
\Phi_0^{0}(r)=&\int_{0}^{\pi} \phi(r, \T)d\T,\\
\Lambda_0^{0}(r)=&\int_{0}^{\pi} \lambda(r,\T)d\T=-\pi\log\left(\dfrac{2(1-\sqrt{1-r^2})}{r^2}\right),\\
\Lambda_1^{0}(r)=&\int_{0}^{\pi} \dfrac{\lambda(r,\T)}{1+r\cos\T}d\T=\dfrac{\pi\log(1-r^2)-\Lambda_0^0(r)}{\sqrt{1-r^2}}.\\
\end{aligned}
\end{equation}
Then, their derivatives can be expressed explicitly as functions of \eqref{eq:12b} and \eqref{eq:12a},
\begin{align*}
\frac{d}{dr}\Phi_0^{0}(r)=&\frac{r^2\,\Phi_0^{0}(r)+\lambda(r,0)-\lambda(-r,0)}{r(1-r^2)}\\
\frac{d}{dr}\Lambda_0^{0}(r)=&-\dfrac{\pi(1-\sqrt{1-r^2})}{r\sqrt{1-r^2}},\\
\frac{d}{dr}\Lambda_1^{0}(r)=&\dfrac{1}{1-r^2}\bigg(r\,\Lambda_1^{0}(r)-\dfrac{(1+r^2-\sqrt{1-r^2})\pi}{r\sqrt{1-r^2}}
\bigg).
\end{align*}
\end{lem}

The following results give recurrent formulas in terms of $k$ and $\ell$ for all the functions defined at the beginning of this section.

\begin{prop}\label{prop:7}
The functions $\mathcal{S}^{\alpha}_{k,\ell}$ and $\mathcal{C}^{\alpha}_{k,\ell},$ defined in~\eqref{eq:11}, write as
\begin{align*}
\mathcal{S}^{\alpha}_{k,\ell}(r,\theta)&=
\begin{cases}
0 & k=0, \ell\ge 0,\\
\big(\cos(k\alpha)-\cos(k(\T+\alpha)\big)k^{-1} & k\ge 1, \ell =0,\\[4pt]
\big(\lambda(r,\alpha)-\lambda(r,\T+\alpha)\big)r^{-1} & k=1,\ell=1,\\[4pt]
\dfrac{(1+r\cos(\T+\alpha))^{1-\ell}-(1+r\cos\alpha)^{1-\ell}}{r(\ell-1)} & k=1,\ell\ge 2,\\[4pt]
2\left(\mathcal{S}^{\alpha}_{k-1,\ell-1}(r,\theta)-
\mathcal{S}^{\alpha}_{k-1,\ell}(r,\theta)\right)r^{-1} - \mathcal{S}^{\alpha}_{k-2,\ell}(r,\theta)
& k\ge 2, \ell\ge 1,
\end{cases}
\\
\mathcal{C}^{\alpha}_{k,\ell}(r,\theta)&=
\begin{cases}
\theta& k=0, \ell=0,\\
\big(\sin(k(\T+\alpha))-\sin(k\alpha)\big)k^{-1}& k\ge 1,\ell=0,\\[4pt]
\phi(r,\alpha)-\phi(r,\theta+\alpha)+\T(1-r^2)^{-\frac{1}{2}} & k=0, \ell =1,\\[4pt]
\mathcal{C}^{\alpha}_{0,\ell-1}(r,\theta) +\dfrac{r}{\ell-1}\dfrac{\partial}{\partial
r}\mathcal{C}^{\alpha}_{0,\ell-1}(r,\theta)& k=0,\ell\ge 2,\\[4pt]
\left(\mathcal{C}^{\alpha}_{0,\ell-1}(r,\theta)
-\mathcal{C}^{\alpha}_{0,\ell}(r,\theta) \right)\,r^{-1}& k=1,\ell\ge 1,\\[4pt]
2\left(\mathcal{C}^{\alpha}_{k-1,\ell-1}(r,\theta)-
\mathcal{C}^{\alpha}_{k-1,\ell}(r,\theta)\right)\, r^{-1} - \mathcal{C}^{\alpha}_{k-2,\ell}(r,\theta)
& k\ge 2, \ell\ge 1,
\end{cases}
\end{align*}
when $r\neq0.$ Furthermore, $\mathcal{S}^{\alpha}_{k,\ell}(0,\theta)=\mathcal{S}^{\alpha}_{k,0}(r,\theta)$ and
$\mathcal{C}^{\alpha}_{k,\ell}(0,\theta)=\mathcal{C}^{\alpha}_{k,0}(r,\theta).$
\end{prop}

\begin{proof}
The expressions of $\mathcal{S}^\alpha_{k,\ell}(r,\theta)$ for $k=0,1$ follow by direct integration. When $k\geq 2$ and $\ell\geq2,$ from its definition and by using elementary transformations, we get
\begin{equation}\label{demint}
\begin{aligned}
\mathcal{S}^\alpha_{k-1,\ell-1}(r,\theta)\!=\!&\int_{\alpha}^{\alpha\!+\!\theta}\frac{\sin((k\!-\!1)\psi)}{(1\!+\!r\cos\psi)^{\ell\!-\!1}}d\psi\!=\!\int_{\alpha}^{\alpha\!+\!\theta}\frac{\sin((k\!-\!1)\psi)(1\!+\!\cos\psi)}{(1\!+\!r\cos\psi)^{\ell}}d\psi\\
&\int_{\alpha}^{\alpha\!+\!\theta}\frac{\sin((k\!-\!1)\psi)}{(1\!+\!r\cos\psi)^{\ell}}d\psi + r\int_{\alpha}^{\alpha\!+\!\theta}\frac{\sin((k\!-\!1)\psi)\cos\psi}{(1\!+\!r\cos\psi)^{\ell}}d\psi.
\end{aligned}
\end{equation}
Using the identity $2\sin((k-1)\psi)\cos\psi=\sin(k\psi)+\sin((k-2)\psi)$ the above expression writes 
\[
\mathcal{S}^\alpha_{k-1,\ell-1}(r,\theta)=\mathcal{S}^\alpha_{k-1,\ell}(r,\theta)+\frac{1}{2}r\,\mathcal{S}^\alpha_{k,\ell}(r,\theta)+ \frac{1}{2}r\,\mathcal{S}^\alpha_{k-2,\ell}(r, \theta).
\]
Then, solving $\mathcal{S}^\alpha_{k,\ell}(r,\theta)$ in this expression we recover the one appearing in the statement.

\medskip

The expression for $\mathcal{C}^\alpha_{k,0}(r,\theta)$ follows by a direct integration, whereas $\mathcal{C}^\alpha_{0,1}$ follows from the definition of $\phi$ given in \eqref{eq:12b} and the change of variables $\tan(\psi/2)=\varphi.$ The expression of $\mathcal{C}^\alpha_{0,\ell}(r,\theta),$ for $\ell\geq2,$ follows deriving with respect to $r.$ The expression for $\mathcal{C}^\alpha_{k,\ell}(r, \theta)$ can be obtained analogously to $\mathcal{S}^\alpha_{k,\ell}(r, \theta).$
\end{proof}

The next corollary follows straightaway by evaluating $\theta=\pi$ in the last result.

\begin{cor}\label{cor:8}
The functions ${s}_{k,\ell}^{\alpha}$ and ${c}_{k,\ell}^{\alpha}$ defined in~\eqref{eq:11} write as
\begin{align*}
s_{k,\ell}^{\alpha}(r)&=
\begin{cases}
0 & k=0,\ell\geq0,\\[4pt]
\left(\lambda(r,\alpha)-\lambda(r,\alpha+\pi)\right)r^{-1}& k=1,\ell= 1,\\[4pt]
\dfrac{(1-(-1)^k)}{k}\cos(k\alpha)& k\ge 1,\ell = 0,\\[4pt]
\dfrac{(1+r\cos \alpha)^{1-\ell}-(1-r\cos \alpha)^{1-\ell}}{r(1-\ell)}& k=1,\ell \ge 2,\\[4pt]
2\,\big(s^{\alpha}_{k-1,\ell-1}(r)-s^{\alpha}_{k-1,\ell}(r)\big)\, r^{-1} - s^{\alpha}_{k-2,\ell}(r)& k\ge 2,\ell \ge 1,
\end{cases}
\\
c_{k,\ell}^{\alpha}(r)&=
\begin{cases}
\pi & k=0,\ell=0,\\[4pt]
\phi(r,\alpha)-\phi(r,\alpha+\pi)+\dfrac{\pi}{\sqrt{1-r^2}}& k=0,\ell=1,\\[4pt]
\dfrac{r}{\ell-1}\dfrac{d}{dr}c^{\alpha}_{0,\ell-1}(r)+c^{\alpha}_{0,\ell-1}(r) & k=0,\ell\geq 2,\\[4pt]
\dfrac{-1+(-1)^k}{k}\sin(k\alpha)& k\ge 1,\ell = 0,\\[4pt]
\big(c^{\alpha}_{0,\ell-1}(r)-c^{\alpha}_{0,\ell}(r)\big)\, r^{-1}& k=1,\ell \ge 1,\\[4pt]
2\,\big(c^{\alpha}_{k-1,\ell-1}(r)-c^{\alpha}_{k-1,\ell}(r)\big)\, r^{-1} - c^{\alpha}_{k-2,\ell}(r)& k\ge 2,\ell \ge 1,\\[4pt]
\end{cases}
\end{align*}
when $r\neq0.$ Furthermore, $s_{0,\ell}^{\alpha}(0)=0,$ $c_{0,\ell}^{\alpha}(0)=\pi$ for $\ell\ge 0$ and 
\[
s_{k,\ell}^{\alpha}(0)=\dfrac{1-(-1)^k}{k}\cos(k\alpha), \quad c_{k,\ell}^{\alpha}(0)=\dfrac{-1+(-1)^k}{k}\sin(k\alpha),
\]
for $k\ge 1$ and $\ell\ge 0.$
\end{cor}

\begin{prop}\label{prop:9l}
The functions $s^{\lambda}_{k,\ell}$ and $c^{\lambda}_{k,\ell},$ defined in~\eqref{eq:11}, write as
\begin{align*}
s^{\lambda}_{k,\ell}(r)&=
\begin{cases}
0& k=0,\ell\geq0,\\[4pt]
\dfrac{\lambda(r,0)-(-1)^k\lambda(-r,0)}{k}+\dfrac{r}{2k}\left(s_{k-1,1}(r)-s_{k+1,1}(r)\right)&
k\geq1,\ell=0,\\[4pt]
\dfrac{\lambda(r,0)^2-\lambda(-r,0)^2}{2r}&
k=1,\ell= 1,\\[4pt]
\dfrac{1}{r(\ell-1)}\Big(\dfrac{\lambda(-r,0)+(\ell-1)^{-1}}{(1-r)^{\ell-1}}-\dfrac{\lambda(r,0)+(\ell-1)^{-1}}{(1+r)^{\ell-1}}\Big)&
k=1,\ell\ge 2,\\[4pt]
2\big(s^{\lambda}_{k-1,\ell-1}(r)-s^{\lambda}_{k-1,\ell}(r)\big)\,r^{-1}-s^
{\lambda}_{k-2,\ell}(r)& k\ge 2,\ell \ge 1,
\end{cases}\\[4pt]
c^{\lambda}_{k,\ell}(r)&=
\begin{cases}
\Lambda_{\ell}^{0}(r)&
k=0,\ell= 0,1,\\[4pt]
\dfrac{r}{\ell-1}\left(\dfrac{d}{dr}\,c^{\lambda}_{0,\ell-1}(r)-c_{1,\ell}
(r)\right)+c^{\lambda}_{0,\ell-1}(r)& k=0,\ell\ge 2,\\[4pt]
\dfrac{r}{2k}\left(c_{k-1,1}(r)-c_{k+1,1}(r)\right)& k\ge
1,\ell=0,\\[4pt]
\left(c^{\lambda}_{0,\ell-1}(r)-c^{\lambda}_{0,\ell}(r)\right)\,r^{-1}&
k=1,\ell\ge 1,\\[4pt]
2\left(c^{\lambda}_{k-1,\ell-1}(r)-c^{\lambda}_{k-1,\ell}(r)\right)\,r^{-1}-c^{\lambda}_{k-2,\ell}(r)& k\ge 2,\ell \ge 1,
\end{cases}
\end{align*}
when $r\neq0.$ Furthermore, $s^{\lambda}_{k,\ell}(0)=s_{k,\ell}(0)$ and $c^{\lambda}_{k,\ell}(0)=0$ for $k\ge 0$ and $\ell\ge 0.$
\end{prop}

\begin{proof}
The general expression for $k\ge,\ell\ge1$ follows similarly to \eqref{demint}. The other cases follow straightforward. In some of them, the integration by parts rule is necessary and also the fact that $c^{\lambda}_{k,\ell}$ is an even function.
\end{proof}

\begin{prop}\label{prop:10}
The functions $s^{\phi}_{k,\ell}$ and $c^{\phi}_{k,\ell},$ defined in~\eqref{eq:11}, write as
\begin{align*}
s^{\phi}_{k,\ell}(r)&=
\begin{cases}
0& k=0,\ell\ge 0,\\[4pt]
-\dfrac{c_{k,1}(r)}{k}
& k\ge 1,\ell=0,\\[4pt]
\dfrac{\Lambda^{0}_0(r)}{r\sqrt{1-r^2}}-\dfrac{\Lambda^{0}_1(r)}{r}
&k=1,\ell=1,\\[4pt]
\dfrac{1}{r(\ell-1)}\Big(c_{0,\ell}-\dfrac{c_{0,\ell-1}}{\sqrt{1-r^2}}\Big)
& k=1,\ell \ge 2,\\[4pt]
2\left(s^{\phi}_{k-1,\ell-1}(r)-s^{\phi}_{k-1,\ell}(r)\right) \, r^{-1}-s^{\phi}_{k-2,\ell}(r)& k\ge
2,\ell \ge 1,
\end{cases}\\[4pt]
c^{\phi}_{k,\ell}(r)&=
\begin{cases}
\Phi_0^{0}(r)&
k=0,\ell= 0,\\[4pt]
\dfrac{c_{0,1}^{\T}(r)}{\sqrt{1-r^2}} -\dfrac{1}{2}\left(\dfrac{\pi}{\sqrt{1-r^2}}\right)^2
& k=0,\ell=1,\\[4pt]
\dfrac{r}{\ell-1}\left(\dfrac{d}{dr}c_{0,\ell-1}^{\phi}(r)-\dfrac{s_{1,\ell}(r)}{1-r^2}\right)+\dfrac{\ell(1-r^2)-1}{(\ell-1)(1-r^2)}c_{0,\ell-1}^{\phi}(r)& k=0,\ell\geq2,\\
\dfrac{s_{k,1}(r)}{k}+\dfrac{(-1)^k-1}{k^2\sqrt{1-r^2}}& k\geq 1,\ell= 0,\\[4pt]
\left(c^{\phi}_{0,\ell-1}(r)-c^{\phi}_{0,\ell}(r)\right)\,r^{-1}&
k=1,\ell\ge 1,\\[4pt]
2\left(c^{\phi}_{k-1,\ell-1}(r)-c^{\phi}_{k-1,\ell}(r)\right)\,r^{-1}-c^
{\phi}_{k-2,\ell}(r)& k\ge 2,\ell \ge 1,
\end{cases}
\end{align*}
when $r\neq0.$ Furthermore, $s^{\phi}_{k,\ell}(0)=c^{\phi}_{k,\ell}(0)=0.$
\end{prop}

\begin{proof} 
The expression of $s^{\phi}_{k,\ell}$ follows from the fact that it is an even function and the results in \cite{ProTor2014}. For $k=1,\ell\ge1$ and $k\ge2,\ell\ge1,$ proceeding analogously to the previous proofs, we get $c^{\phi}_{k,\ell}.$ Finally, for $k=0,\ell\ge2,$ we compute directly the derivative, with respect to $r,$ of $c^{\phi}_{0,\ell-1}.$ The other cases follow using the integration by parts rule.
\end{proof}

The last result follows similarly as all the previous results.

\begin{prop}\label{prop:9t}
The functions $s^{\T}_{k,\ell}$ and $c^{\T}_{k,\ell}$ defined in~\eqref{eq:11} write as
\begin{align*}
s^{\T}_{k,\ell}(r)&=
\begin{cases}
0& k=0,\ell\geq0,\\[4pt]
\dfrac{-(-1)^k\pi}k&
k\geq1,\ell=0,\\[4pt]
\big(-\pi\lambda(-r,0)+\Lambda_0^{0}(r)\big)r^{-1}&
k=1,\ell= 1,\\[4pt]
\dfrac{1}{r(\ell-1)}\left(\dfrac{\pi}{(1-r)^{\ell-1}}-c_{0,\ell-1}(r)\right)&
k=1,\ell\ge 2,\\[4pt]
2\Big(s^{\T}_{k-1,\ell-1}(r)-s^{\T}_{k-1,\ell}(r)\Big)\,r^{-1}-s^
{\T}_{k-2,\ell}(r)& k\ge 2,\ell \ge 1,
\end{cases}
\\[4pt]
c^{\T}_{k,\ell}(r)&=
\begin{cases}
\dfrac{\pi^2}{2}&
k=0,\ell= 0,\\[7pt]
\Phi_0^{0}(r)+\dfrac{\pi^2}{2\sqrt{1-r^2}}&k=0,\ell=1,\\[7pt]
\dfrac{r}{\ell-1}\dfrac{d}{dr}c^{\T}_{0,\ell-1}(r)+c^{\T}_{0,\ell-1}(r)& k=0 ,\ell\geq2,\\[4pt]
\dfrac{-1+(-1)^k}{k^2}& k\geq1,\ell=0,\\[4pt]
\Big(c^{\T}_{0,\ell-1}(r)-c^{\T}_{0,\ell}(r)\Big)\,r^{-1}& k=1,\ell\ge 1,\\[4pt]
2\Big(c^{\T}_{k-1,\ell-1}(r)-c^{\T}_{k-1,\ell}(r)\Big)\,r^{-1}-c^{\T}_{k-2,\ell}(r)& k\ge 2,\ell \ge 1,
\end{cases}
\end{align*}
when $r\neq0.$ Furthermore, $s^{\T}_{0,\ell}(0)=0,$ $c^{\T}_{0,\ell}(0)=\pi^2/2$ for $\ell\ge 0$ and
\begin{align*}
s^{\T}_{k,\ell}(0)=\dfrac{-(-1)^k\pi}{k}, \quad
c^{\T}_{k,\ell}(0)=\dfrac{-1+(-1)^k}{k^2}.
\end{align*}
for $k\ge 1$ and $\ell\ge 0.$
\end{prop}

\section*{Acknowledgements}
This work has been realized thanks to the Spanish MINECO MTM2016-77278-P (FEDER) grant, Catalan AGAUR 2017 SGR 1617 grant, and the Brazilian FAPESP 2016/11471-2 and CAPES BEX 13473/13-1 grants.

\bibliographystyle{abbrv}
\bibliography{biblio}

\end{document}